\documentclass[12pt]{article}

\usepackage{amsthm}
\usepackage{amssymb}
\usepackage{amsmath}
\usepackage{amsfonts}
\usepackage{times}
\usepackage{bbm}
\usepackage{cite}
\usepackage{color, soul}
\usepackage{hyperref}
\hypersetup{
     colorlinks   = true,
     citecolor    = black,
     linkcolor    = blue
}

\usepackage[paper=a4paper, top=2.5cm, bottom=2.5cm, left=2.5cm, right=2.5cm]{geometry}

\def\qed{\hfill $\vcenter{\hrule height .3mm
\hbox {\vrule width .3mm height 2.1mm \kern 2mm \vrule width .3mm
height 2.1mm} \hrule height .3mm}$ \bigskip}

\def \Sph{\mathbb{S}^{n-1}}
\def \RR {\mathbb R}
\def \NN {\mathbb N}
\def \EE {\mathbb E}

\def \PP {\mathbb P}
\def \eps {\varepsilon}
\def \vphi {\varphi}

\def \re {\mathrm{Re}}

\newcommand\norm[1]{\left\lVert#1\right\rVert}
\newtheorem{theorem}{Theorem}
\newtheorem{lemma}{Lemma}

\newtheorem{claim}[lemma]{Claim}
\newtheorem{conjecture}{Conjecture}

\theoremstyle{definition}

\theoremstyle{remark}

\long\def\symbolfootnotetext[#1]#2{\begingroup
\def\thefootnote{\fnsymbol{footnote}}\footnotetext[#1]{#2}\endgroup}

\makeatletter
\newcommand*{\mint}[1]{%
  \mint@l{#1}{}%
}
\newcommand*{\mint@l}[2]{%
  \@ifnextchar\limits{%
    \mint@l{#1}%
  }{%
    \@ifnextchar\nolimits{%
      \mint@l{#1}%
    }{%
      \@ifnextchar\displaylimits{%
        \mint@l{#1}%
      }{%
        \mint@s{#2}{#1}%
      }%
    }%
  }%
}
\newcommand*{\mint@s}[2]{%
  \@ifnextchar_{%
    \mint@sub{#1}{#2}%
  }{%
    \@ifnextchar^{%
      \mint@sup{#1}{#2}%
    }{%
      \mint@{#1}{#2}{}{}%
    }%
  }%
}
\def\mint@sub#1#2_#3{%
  \@ifnextchar^{%
    \mint@sub@sup{#1}{#2}{#3}%
  }{%
    \mint@{#1}{#2}{#3}{}%
  }%
}
\def\mint@sup#1#2^#3{%
  \@ifnextchar_{%
    \mint@sup@sub{#1}{#2}{#3}%
  }{%
    \mint@{#1}{#2}{}{#3}%
  }%
}
\def\mint@sub@sup#1#2#3^#4{%
  \mint@{#1}{#2}{#3}{#4}%
}
\def\mint@sup@sub#1#2#3_#4{%
  \mint@{#1}{#2}{#4}{#3}%
}
\newcommand*{\mint@}[4]{%
  \mathop{}%
  \mkern-\thinmuskip
  \mathchoice{%
    \mint@@{#1}{#2}{#3}{#4}%
        \displaystyle\textstyle\scriptstyle
  }{%
    \mint@@{#1}{#2}{#3}{#4}%
        \textstyle\scriptstyle\scriptstyle
  }{%
    \mint@@{#1}{#2}{#3}{#4}%
        \scriptstyle\scriptscriptstyle\scriptscriptstyle
  }{%
    \mint@@{#1}{#2}{#3}{#4}%
        \scriptscriptstyle\scriptscriptstyle\scriptscriptstyle
  }%
  \mkern-\thinmuskip
  \int#1%
  \ifx\\#3\\\else_{#3}\fi
  \ifx\\#4\\\else^{#4}\fi  
}
\newcommand*{\mint@@}[7]{%
  \begingroup
    \sbox0{$#5\int\m@th$}%
    \sbox2{$#5\int_{}\m@th$}%
    \dimen2=\wd0 %
    \let\mint@limits=#1\relax
    \ifx\mint@limits\relax
      \sbox4{$#5\int_{\kern1sp}^{\kern1sp}\m@th$}%
      \ifdim\wd4>\wd2 %
        \let\mint@limits=\nolimits
      \else
        \let\mint@limits=\limits
      \fi
    \fi
    \ifx\mint@limits\displaylimits
      \ifx#5\displaystyle
        \let\mint@limits=\limits
      \fi
    \fi
    \ifx\mint@limits\limits
      \sbox0{$#7#3\m@th$}%
      \sbox2{$#7#4\m@th$}%
      \ifdim\wd0>\dimen2 %
        \dimen2=\wd0 %
      \fi
      \ifdim\wd2>\dimen2 %
        \dimen2=\wd2 %
      \fi
    \fi
    \rlap{%
      $#5%
        \vcenter{%
          \hbox to\dimen2{%
            \hss
            $#6{#2}\m@th$%
            \hss
          }%
        }%
      $%
    }%
  \endgroup
}
\begin{document}
\title{Information and Dimensionality of Anisotropic Random Geometric Graphs}
\author{
	Ronen Eldan\\
	\small{Weizmann Institute of Science}
	\and
	Dan Mikulincer\\
	\small{Weizmann Institute of Science}
}
\date{}
\maketitle
\begin{abstract}
	This paper deals with the problem of detecting non-isotropic high-dimensional geometric structure in random graphs. Namely, we study a model of a random geometric graph in which vertices correspond to points generated randomly and independently from a non-isotropic $d$-dimensional Gaussian distribution, and two vertices are connected if the distance between them is smaller than some pre-specified threshold. We derive new notions of dimensionality which depend upon the eigenvalues of the covariance of the Gaussian distribution. If $\alpha$ denotes the vector of eigenvalues, and $n$ is the number of vertices, then the quantities $\left(\frac{\norm{\alpha}_2}{\norm{\alpha}_3}\right)^6/n^3$ and $\left(\frac{\norm{\alpha}_2}{\norm{\alpha}_4}\right)^4/n^3$ determine upper and lower bounds for the possibility of detection. This generalizes a recent result by Bubeck, Ding, R\'acz and the first named author from \cite{BDER14} which shows that the quantity $d/n^3$ determines the boundary of detection for isotropic geometry. Our methods involve Fourier analysis and the theory of characteristic functions to investigate the underlying probabilities of the model. The proof of the lower bound uses information theoretic tools, based on the method presented in ~\cite{BG15}.
\end{abstract}
\centerline{\emph{Keywords:} geometric graphs, random graphs, Erd\H os-R\'enyi model, dimensionality of embeddings}

\section{Introduction}
This study continues a line of work initiated by Bubeck, Ding, R\'acz and the first named author \cite{BDER14}, in which the problem of detecting geometric structure in large graphs was studied. In other words, given a large graph one is interested in determining whether or not it was generated using a latent geometric structure. The main contribution of this study is a generalization of the results to the anisotropic case.\\

Extracting information from large graphs is an extensively studied statistical task. In many cases, a given network, or graph, reflects some underlying structure; for example, a biological neuronal network is likely to reflect certain characteristics of its functionality such as physical location and cell structure. The objective of this paper is thus the detection of such an underlying geometric structure. \\

As a motivating example, consider the graph representing a large social network. It may be assumed that each node (or user) is described by a set of numerical parameters representing its properties (such as geographical location, age, political association, interests, etc). It is plausible to assume that two nodes are more likely to be connected when their two respective points in parameter space are more correlated. Adopting this assumption, the nodes of such a graph may be thought of as points in a Euclidean space, with links appearing between two nodes when their distance is small enough. A natural question in this context would be: What can be said about the geometric structure by inspection of the graph itself? Specifically, can one distinguish between such a graph and a graph with no underlying geometric structure? \\

In statistical terms, given a graph $G$ on $n$ vertices, our null hypothesis is that $G$ is an instance of the standard Erd\H os-R\'enyi random graph $G(n,p)$ \cite{ER60}, where the presence of each edge is determined independently, with probability $p$: 
$$H_0: G \sim G(n,p).$$
On the other hand, for the alternative, we consider the so-called random geometric graph. In this model each vertex is a point in some metric space and an edge is present between two points if the distance between them is smaller than some predefined threshold. Perhaps the most well-studied setting of this model is the isotropic Euclidean model, where the vertices are generated uniformly on the $d$-dimensional sphere or simply from the standard normal $d$-dimensional distribution. However, it seems that this model is too simplistic to reflect real world social networks.
One particular problem, which we intend to tackle in this study, is the isotropicity assumption, which amounts to the fact that all of the properties associated with a node have the same significance in determining the network structure. It is clear that some parameters, such as geographic location, can be more significant than others. We therefore propose to extend this model to a non-isotropic setting. Roughly speaking, we replace the sphere with an ellipsoid; Instead of generating vertices from $\mathcal{N}(0,\mathrm{I}_n)$, they will be generated from  $\mathcal{N}(0,D_\alpha)$ for some diagonal matrix $D_\alpha$ with non-negative entries. We denote the model by $G(n,p,\alpha)$ where $p$ is the probability of an edge appearing, and the diagonal of $D_\alpha$ is given by a vector $\alpha \in \RR^d$. 
Formally, let $X_1,...,X_n$ be i.i.d points generated from $\mathcal{N}(0,D_\alpha)$. In $G(n,p,\alpha)$ vertices correspond to $X_1,...,X_n$ and two distinct vertices are joined by an edge if and only if $\left\langle X_i, X_j\right\rangle \geq t_{p,\alpha}$, where $t_{p,\alpha}$ is the unique number satisfying $\PP(\left\langle X_1,X_2\right\rangle \geq t_{p,\alpha}) = p$. Our alternative hypothesis is thus
$$H_1: G \sim G(n,p, \alpha).$$

In this paper, we will focus on the high-dimensional regime of the problem. Namely, we assume that the dimension and covariance matrix can depend on $n$. This point of view becomes highly relevant when considering recent developments in data sciences, where big data and high-dimensional feature spaces are becoming more prevalent. We will focus on the dense regime, where $p$ is a constant independent of $n$ and $\alpha$.\\
\subsection{Previous work}
This paper can be seen a direct follow-up of ~\cite{BDER14}, which, as noted above, deals with the isotropic model of $G(n,p,d)$ in which $D_\alpha = \mathrm{I}_d$. In the dense regime, it was shown that the total variation between the models depends asymptotically on the ratio $\frac{d}{n^3}$. The dependence is such that if $d >> n^3$, then $G(n,p,d)$ converges in total variation to $G(n,p)$. Conversely, on the other hand, if $d << n^3$ the total variation converges to 1.\\

Our starting point is thus the result of ~\cite{BDER14} stated as follows:
\begin{theorem} \label{thm: BDER}
	(a) Let $p \in (0,1)$ be fixed and assume that $d/n^3 \to 0$. Then,
	$$\mathrm{TV}(G(n,p), G(n,p,d)) \to 1.$$
	(b) Furthermore, if $d/n^3 \to \infty$ then
	$$\mathrm{TV}(G(n,p), G(n,p,d)) \to 0.$$
\end{theorem}

One of the fundamental differences between $G(n,p)$ and $G(n,p,d)$ is a consequence of the triangle inequality. That is, if two points $u$ and $v$ are both close to a point $w$, then $u$ and $v$ cannot be too far apart. This roughly means that if both $u$ and $v$ are connected to $w$, then there is an increased probability of $u$ being connected to $v$, unlike the case of the Erd\H os-R\'enyi graph where there is no dependence between the edges. Thus, counting the number of triangles in a graph seems to be a natural test to uncover geometric structure.\\

The idea of using triangles was extended  in \cite{BDER14} and a variant was proposed: the {\it signed triangle}. This statistic was successfully used to completely characterize the asymptotics of $\mathrm{TV}(G(n,p), G(n,p,d))$ in the isotropic case. To understand the idea behind signed triangles, we first note that if $A$ is the adjacency matrix of $G$ then the number of triangles in $G$ is given by $\mathrm{Tr(A^3)}$. The "number" of signed triangles is represented by $\mathrm{Tr}((A-p \mathbf{1})^3)$ where $\mathbf{1}$ is the matrix whose entries are all equal to $1$. It turns out that the variance of signed triangles is significantly smaller than the corresponding quantity for regular triangles. \\

The methods used in \cite{BDER14} relied heavily on the symmetries of the sphere. As mentioned, our goal is to generalize this to the non-isotropic case, which requires us to apply different methods. The dimension $d$ of the isotropic space arises as a natural parameter  when discussing the underlying probabilities of Theorem \ref{thm: BDER}. Clearly, however, when different coordinates of the space have different scales, the dimension by itself has little meaning. For example, consider a $d$-dimensional ellipsoid with one axis being large and the rest being much smaller. This ellipsoid behaves more like a $1$-dimensional sphere rather than a $d$-dimensional one, in the sense mentioned above. It would stand to reason the more anisotropic the ellipsoid is, the smaller its effective dimension would be. \\

\subsection{Main results and ideas}

In accordance to the above, our first task is to find a suitable notion of dimensionality for our model. For $q \geq1$, let $\norm{\cdot}_q$ stand for the $q$-norm. We derive the quantities $\left(\frac{\norm{\alpha}_2}{\norm{\alpha}_3}\right)^6$ and $\left(\frac{\norm{\alpha}_2}{\norm{\alpha}_4}\right)^4$ as the new notions of  dimension, where $\alpha$ parametrizes the eigenvalues of $D_\alpha$, the covariance matrix of the normal distribution, and is considered as a $d$-dimensional vector.  We note that, in the isotropic case, those quantities reduce to $d$ which also maximizes the expressions.\\

This notion of dimension allows us to tackle the main objective of this paper; studying the total variation distance between $G(n,p)$ and $G(n,p,\alpha)$. Considering what we know about the isotropic case our question becomes: What conditions are required from $\alpha$, so that the total variation remains bounded away from 0?
The following theorem provides a sufficient condition on $\alpha$ as well as a necessary one:
\begin{theorem}{\it (a)} \label{thm: main}
	Let $p \in (0,1)$ be fixed and assume that $\left(\frac{\norm{\alpha}_2}{\norm{\alpha}_3}\right)^6/n^3 \to 0$. Then, $$\mathrm{TV}(G(n,p),G(n,p,\alpha)) \to 1.$$ 
	{\it (b)} Furthermore, if $\left(\frac{\norm{\alpha}_2}{\norm{\alpha}_4}\right)^4/n^3 \to \infty$, then
	$$\mathrm{TV}(G(n,p),G(n,p,\alpha)) \to 0.$$ 
\end{theorem}
Note that there is a gap between the bounds 2(a) and 2(b) (for example, if $\alpha_i \sim \frac{1}{\sqrt[3]{i}}$, then $\left(\frac{\norm{\alpha}_2}{\norm{\alpha}_3}\right)^6$ is order of $\frac{d}{\ln^2(d)}$, while $\left(\frac{\norm{\alpha}_2}{\norm{\alpha}_4}\right)^4$ is about $d^{\frac{2}{3}}$). We conjecture that the bound 2(a) is tight:
\begin{conjecture} \label{conj}
	Let $p \in (0,1)$ be fixed and assume that $\left(\frac{\norm{\alpha}_2}{\norm{\alpha}_3}\right)^6/n^3 \to \infty$. Then $$\mathrm{TV}(G(n,p),G(n,p,\alpha)) \to 0$$ 
\end{conjecture}
In the following we describe some of the ideas used to prove Theorem \ref{thm: main}.\\

As discussed, the main idea underlying this work has to do with counting triangles. Given a graph $G$ we denote by $T(G)$ the number of triangles in the graph. It is easy to verify that $\EE\ T(G(n,p)) = \binom{n}{3}p^3$ and $\mathrm{Var}(T(G(n,p)))$ is of order $n^4$. In the isotropic case, standard calculations show that the expected number of triangles in $G(n,p,d)$ is boosted by a factor proportional to $1+ \frac{1}{\sqrt{d}}$. The first difficulty that arises is to find a precise estimate for the probability increment in the non-isotropic case. In this case, we show that there is a constant $\delta_p$ depending only on $p$ such that $\EE\ T(G(n,p,\alpha)) \geq \binom{n}{3}p^3\left(1+ \delta_p\left(\frac{\norm{\alpha}_3}{\norm{\alpha}_2}\right)^3\right)$. This would imply a non-negligible total variation distance as long as $\binom{n}{3}\left(\frac{\norm{\alpha}_3}{\norm{\alpha}_2}\right)^3$ is bigger than the standard deviation of $T(G(n,p))$. We incorporate the idea of using {\it signed triangles} which attain a similar difference between expected values but have a smaller variance. The number of {\it signed triangles} is defined as:
$$\tau(G) =  \sum\limits_{\{i,j,k\}\in\binom{[n]}{3}}(A_{i,j}-p)(A_{i,k}-p)(A_{j,k}-p),$$
where $A$ is the adjacency matrix of $G$, which is proportional to $\mathrm{Tr}((A-p\mathbf{1})^3)$. It is known that $\mathrm{Var}(\tau(G(n,p)))$ is only of order $n^3$. Resolving the value of $\mathrm{Var}(\tau(G(n,p,\alpha)))$ leads to the following result (which implies Theorem \ref{thm: main}(a)):
\begin{theorem} \label{thm: signed triangles}
	Let $p \in (0,1)$ be fixed and assume that $\left(\frac{\norm{\alpha}_2}{\norm{\alpha}_3}\right)^6/n^3 \to 0$. Then $$\mathrm{TV}(\tau(G(n,p)),\tau(G(n,p,\alpha))) \to 1$$ 
\end{theorem}
To prove Theorem \ref{thm: main}(b) we may view the random graph $G(n,p,\alpha)$ as a measurable function of a random $n \times n$ matrix $W(n,\alpha)$ with entries proportional to $\left\langle\gamma_i, \gamma_j \right\rangle$ where $\gamma_i$ are drawn {\it i.i.d} from $\mathcal{N}(0, D_\alpha)$ and $D_\alpha = \mathrm{diag}(\alpha)$. Similarly, $G(n,p)$ can be viewed as a function of an $n\times n$ GOE random matrix denoted by $M(n)$. In ~\cite{BDER14} Theorem \ref{thm: BDER}(b) was proven using direct calculations on the densities of the involved distributions. However, in our case, no simple formula exists, which makes their method inapplicable. The premise is instead proven using information theoretic tools, adopting ideas from \cite{BG15}. The main idea is to use Pinsker's inequality to bound the total variation distance by the respective relative entropy. Thus we are interested in
$$\mathrm{Ent}\left[W(n,\alpha)||M(n)\right].$$
Theorem \ref{thm: main}(b) will then follow from the next result:
\begin{theorem} \label{thm: entropy}
	Let $p \in (0,1)$ be fixed and assume that $\left(\frac{\norm{\alpha}_2}{\norm{\alpha}_4}\right)^4/n^3 \to \infty$. Then $$\mathrm{Ent}\left[W(n,\alpha)||M(n)\right] \to 0.$$ 
\end{theorem} 
We suspect, as stated in Conjecture \ref{conj}, that Theorem \ref{thm: main}(b) does not give a tight characterization of the lower bound. Indeed, in the dense regime of the isotropic case, {\it signed triangles} act as an optimal statistic. It would seem to reason that deforming the sphere shouldn't affect the utility of such a local tool.

\paragraph{Acknowledgments:}We would like to thank the anonymous referee for carefully reading this paper and for the thoughtful comments which helped improve the overall presentation.

\section{Preliminaries}
We work in $\RR^n$, equipped with the standard Euclidean structure $\langle \cdot, \cdot \rangle$. For $q \geq 1$, we denote by $\norm{\cdot}_q$ the corresponding $q$-norm. That is, for $(v_1,...,v_n)=v \in \RR^n$, $\norm{v}_q = \left(\sum\limits_{i=1}^n |v_i|^q\right)^{\frac{1}{q}}$. If $\alpha = \{\alpha_i\}_{i=1}^d$ is a multi-set with elements from $\RR$, we adopt the same notation for $\norm{\alpha}_q$. We abbreviate $\norm{\cdot}:=\norm{\cdot}_2 $, the usual Euclidean norm and denote by $\Sph$ the unit sphere under this norm. In our proofs, we will allow ourselves to use the letters $c, C,c',C', c_1,C_1$, etc. to denote absolute positive constants whose values may change between appearances. The letters $x,y,z$ will usually denote spatial variables while $a,b,c$ will denote the corresponding frequencies in the Fourier domain. The letters $X,Y,Z$ will usually be used as random variables and vectors. The imaginary unit will be denoted as $\mathrm{\bf{i}}$.\\
\\
Let $X$ be a real valued random variable. The characteristic function of $X$ is a function $\varphi_X: \RR \to \RR$, given by
$$\varphi_X(t) = \EE[e^{\mathrm{\bf{i}}tX}].$$
More generally, if $X$ is an $n$-dimensional random vector, then the characteristic function of $X$ is a function $\varphi_X: \RR^n \to \RR$ given by  
$$\varphi_X(t) = \EE[e^{\mathrm{\bf{i}}\langle t,X\rangle}].$$
By elementary Fourier analysis, one can use the characteristic function to recover the distribution, whenever the random vector is integrable. We will be interested in the specific case where the dimension of $X$ is $3$. Assume $X = (X^1,X^2,X^3)$ has a density, denoted by $f$, a characteristic function, denoted by $\varphi$ and (reflected) cumulative distribution function
$$F(t_1,t_2,t_3) = \PP(X^1>t_1, X^2>t_2, X^3>t_3),$$
with marginals onto the first 1 or 2 coordinates denoted as $F(t_1,t_2)$ and $F(t_1)$ respectively (remark that this is slightly a different definition than the usual notion of the cumulative distribution function, which results in a change of sign for the next identity). Then e.g., \cite[Theorem 5]{S91} states that
\begin{align} \label{invere char function}
\frac{\mathrm{\bf{i}}}{\pi^3}\mint{\times}\limits_{\RR^3}&\frac{\varphi(a,b,c)e^{-\mathrm{\bf{i}}(at_1+bt_2+ct_3)}}{abc}dadbdc =\\ &8F(t_1,t_2,t_3)-4(F(t_1,t_2)+F(t_2,t_3)+F(t_1,t_3))+2(F(t_1)+F(t_2)+F(t_3))-1, \nonumber
\end{align}
where the integral is taken as a Cauchy principal value; In $\RR^3$, the Cauchy principal value of a function $g$, which we henceforth denote by $\mint{\times}\limits_{\RR^3}g$, is defined as
$$\int\limits_0^{\infty}\int\limits_0^{\infty}\int\limits_0^{\infty}\Delta_c\Delta_b\Delta_a g(a,b,c)dadbdc,$$
where $\Delta_ag(a,b,c) := g(a,b,c)+g(-a,b,c)$ and likewise for $b,c$. In the following, for multivariate functions, we interpret the definition of an \emph{odd} (resp. \emph{even}) function in the following sense: $g$ is odd (resp. even) if it is antisymmetric (resp. symmetric) under change of sign of any coordinate, while keeping the values of the rest of the coordinates intact. We note that the principal value of an odd function vanishes, and if $g$ is integrable then $\mint{\times}\limits_{\RR^3}g = \int\limits_{\RR^3}g$. Furthermore, by denoting
$$\mathrm{sgn}_{(t_1,t_2,t_3)}(x,y,z) = \mathrm{sgn}(x-t_1)\mathrm{sgn}(y-t_2)\mathrm{sgn}(z-t_3),$$
a simple calculation shows the following equality:
\begin{align*}
\int\limits_{\RR^3}&f(x,y,z)\cdot\mathrm{sgn}_{(t_1,t_2,t_3)}(x,y,z)dxdydz = \\
&8F(t_1,t_2,t_3)-4(F(t_1,t_2)+F(t_2,t_3)+F(t_1,t_3))+2(F(t_1)+F(t_2)+F(t_3))-1.
\end{align*}
Since the Fourier transform is an isometry we have that  
\begin{equation} \label{Fourier isometry}
\int\limits_{\RR^3}f\cdot\mathrm{sgn}_{(t_1,t_2,t_3)} =\frac{1}{\pi^3} \mint{\times}\limits_{\RR^3}\varphi\cdot\widehat{\mathrm{sgn}}_{(t_1,t_2,t_3)}, 
\end{equation}
where $\widehat{\mathrm{sgn}}_{(t_1,t_2,t_3)}$ is the Fourier transform of $\mathrm{sgn}_{(t_1,t_2,t_3)}$, when considered as a tempered distribution (for more information on the topic, see \cite{hormander2015analysis}).
Putting all of the above together yields
\begin{equation} \label{shifted sign}
\widehat{\mathrm{sgn}}_{(t_1,t_2,t_3)}(a,b,c)= \frac{\mathrm{\bf{i}}e^{-\mathrm{\bf{i}}(at_1+bt_2+ct_3)}}{abc}.
\end{equation}
\\
For a positive semi-definite $n \times n$ matrix $\Sigma$, we denote by $\mathcal{N}(0,\Sigma)$ the law of the centered Gaussian distribution with covariance $\Sigma$. If $X \sim \mathcal{N}(0,\Sigma)$ then $X^TX$ has the law $\mathcal{W}_n(\Sigma, 1)$ of the Wishart distribution with $1$ degree of freedom. The characteristic function of $X^TX$ is known (see ~\cite{eaton}) and given by
\begin{equation} \label{char of wishart}
\Theta \to \det\left(\mathrm{I} - 2\mathrm{\bf{i}}\Theta\Sigma\right)^{-\frac{1}{2}}.
\end{equation}
If $Z$ is distributed as a standard Gaussian random variable, then $Z^2$ has the $\chi^2$ distribution with $1$ degree of freedom. For such a distribution, we have $\EE[\chi^2] = 1$ and $\mathrm{Var}(\chi^2)=2$. The $\chi^2$ distribution has a sub-exponential tail which may be bounded using a Bernstein's type inequality (~\cite{V12}), in the following way. If $\{\chi^2_i\}_{i=1}^n$, are independent $\chi^2$ random variables, then for every $(v_1,...,v_n) = v \in \RR^n$ and every $t>0$
\begin{equation} \label{berenstein type}
\PP\left(\left|\sum v_i\chi_i^2-\sum v_i\right|\geq t\right) \leq 2\exp\left(-\min\left(\frac{t}{2\norm{v}_{\infty}},\frac{t^2}{4\norm{v}^2_{2}}\right)\right).
\end{equation}
Let $X_1,...,X_n$ be independent random variables with $0$ mean and variance $\EE[X_i^2] = \sigma_i^2$. Define $s_n = \sqrt{\sum\limits_{i=1}^n \sigma_i^2}$ and $S_n = \sum\limits_{i=1}^n\frac{X_i}{s_n}$. Under appropriate regularity conditions the central limit theorem states that $S_n$ converges in distribution to $\mathcal{N}(0,1)$.\\
\\
Berry-Esseen's inequality ~\cite{P95} quantifies this convergence. 
Suppose that the absolute third moments of $X_i$ exist and $\EE[|X_i|^3] = \rho_i$. If we denote by $Z$ a standard Gaussian and define $S_n$ as above then, for every $x \in \RR$,
\begin{equation} \label{prelim berry eseen}
|\PP(S_n < x) - \PP(Z<x)| \leq \frac{\sum\limits_{i=1}^n\rho_i}{s_n^3}.
\end{equation}
This can be generalized to higher dimensions, as found in ~\cite[Theorem 1.1]{bentkus2005lyapunov}. In that case assume $X_1,...,X_n$ are independent random vectors in $\RR^d$ and $S_n = \sum\limits_{i=1}^{n}X_i$ has covariance $\Sigma^2$. Assume that $\Sigma$ is invertible and denote $\EE[|\Sigma^{-1}X_i|^3]=\rho_i$. If $Z_d$ is a $d$-dimensional standard Gaussian vector, then there exists a universal constant $C_{\mathrm{be}}>0$, such that for any convex set $A$:
\begin{equation} \label{multi-berry-eseen}
|\PP(\Sigma^{-1}S_n \in A) - \PP(Z_d \in A)| \leq C_{\mathrm{be}}d^{\frac{1}{4}}\sum\limits_i\rho_i.
\end{equation}
For a random vector $X$ on $\RR^n$ with density $f$, the differential entropy of $X$ is defined
$$\mathrm{Ent}[X] =  -\int\limits_{\RR^n}f(x)\ln(f(x))dx.$$
If $Y$ is another random vector with density $g$, the relative entropy of $X$ with respect to $Y$ is 
$$\mathrm{Ent}[X||Y] = \int\limits_{\RR^n}f(x)\ln\left(\frac{f(x)}{g(x)}\right)dx.$$
Pinsker's inequality connects between the relative entropy and the total variation distance,
\begin{equation} \label{pinsker}
\mathrm{TV}(X,Y) \leq \sqrt{\frac{1}{2}\mathrm{Ent}[X||Y]}.
\end{equation}
The chain rule for relative entropy states that for any random vectors $X_1,X_2,Y_1,Y_2$,
\begin{equation} \label{chain rule}
\mathrm{Ent}[(X_1,X_2)||(Y_1,Y_2)] = \mathrm{Ent}[X_1||Y_1]+\EE_{x\sim \lambda_1}\mathrm{Ent}[X_2|X_1=x||Y_2|Y_1 = x],
\end{equation}
where $\lambda_1$ is the marginal of $X_1$, and $X_2|X_1=x$ is the distribution of $X_2$ conditioned on the event $X_1 = x$ (similarly for $Y_2|Y_1 = x$).
\section{Estimates for a triangle in a random geometric graph}
In this section we derive a lower bound for the probability that an induced subgraph, of size 3, of a random geometric graph forms a triangle. This calculation is instrumental for the derivation of Theorem \ref{thm: main}(a). Using the notation of the introduction, let $X_1,X_2,X_3 \sim \mathcal{N}(0,D_{\alpha})$ be independent normal random vectors with coordinates $X_1^i,X_2^i,X_3^i$ for $1 \leq i \leq d$. We denote by $f$ the joint density of $(\langle X_1, X_2\rangle, \langle X_1, X_3\rangle,\langle X_2, X_3\rangle)$. Consider the event
$$E_p = \{\langle X_1, X_2\rangle \geq t_{p,\alpha},\langle X_1, X_3\rangle \geq t_{p,\alpha},\langle X_2, X_3\rangle \geq t_{p,\alpha}\},$$
that the corresponding vertices form a triangle in $G(n,p,\alpha)$.
The main result of this section is the following theorem.
\begin{theorem} \label{thm: prob estimation}
Let $p \in (0,1)$ and assume $\norm{\alpha}_\infty=1$. One has 
$$p^3+\Delta\left(\frac{\norm{\alpha}_3}{\norm{\alpha}_2}\right)^3 \geq \PP(E_p) \geq p^3+\delta_p\left(\frac{\norm{\alpha}_3}{\norm{\alpha}_2}\right)^3$$
whenever $\norm{\alpha}_2 > c_p$, for constants $\Delta,\delta_p,c_p >0$ which may depend only on $p$.
\end{theorem}
\subsection{Lower bound; the case $p=\frac{1}{2}$}
It will be instructive to begin the discussion with the (easier) case $p = \frac{1}{2}$, in which $t_{p,\alpha} = 0$. We are thus interested in the probability that $\langle X_1, X_2\rangle,\langle X_1, X_3\rangle,\langle X_2, X_3\rangle > 0$. Note that the triplet $(\langle X_1, X_2\rangle, \langle X_1, X_3\rangle,\langle X_2, X_3\rangle)$ can be realized as a linear combination of upper off-diagonal elements taken from $d$ independent $3$-dimensional Wishart random matrices (see below for an elaborated explanation). Unfortunately, there is no known closed expression for the density of such a distribution. The following lemma utilizes the characteristic function of the joint distribution to derive a closed expression for the desired probability.   \\
\begin{lemma} \label{lem: into char function}
\begin{equation} \label{triangle prob}
\PP \left (E_{\frac{1}{2}} \right )= \frac{1}{8}+\mint{\times}\limits_{\mathbb{R}^3}\frac{\mathrm{\bf{i}}}{8abc\pi^3} \left(\prod\limits_i(1 +\alpha_i^2(a^2 +b^2+c^2) + 2\alpha_i^3abc\mathrm{\bf{i}})^{-\frac{1}{2}}\right) \ dadbdc.
\end{equation}
\end{lemma}

\begin{proof}
Consider the event $\{\langle X_1, X_2\rangle > 0, \langle X_1, X_3\rangle < 0,\langle X_2, X_3\rangle<0 \}$. The map $(x,y,z) \mapsto (x, y, -z)$ preserves the law of $(X_1,X_2,X_3)$. Thus, 
\begin{flalign*}
\mathbb{P}(\{\langle X_1, X_2\rangle > 0,& \langle X_1, X_3\rangle < 0,\langle X_2, X_3\rangle<0 \}) &&\\&= \mathbb{P}(\{\langle X_1, X_2\rangle > 0, \langle X_1, X_3\rangle > 0,\langle X_2, X_3\rangle>0 \}).&&
\end{flalign*} By the same argument, 
\begin{flalign*}
\mathbb{P}(\{\langle X_1, X_2\rangle > 0,& \langle X_1, X_3\rangle > 0,\langle X_2, X_3\rangle<0 \})&&\\ &= \mathbb{P}(\{\langle X_1, X_2\rangle < 0, \langle X_1, X_3\rangle < 0,\langle X_2, X_3\rangle<0\}).&&
\end{flalign*}
We denote the event on the right side by $\PP \left (I_{\frac{1}{2}} \right )$, the probability of an induced independent set on $3$ vertices.\\

From the above observation, it is clear that $4\left(\PP(E_{\frac{1}{2}}) + \PP(I_{\frac{1}{2}})\right) = 1$. Also, we may note that $\int\limits_{\mathbb{R}^3} \mathrm{sgn}(xyz)\cdot f(x,y,z)\ dxdydz = 4\left(\PP(E_{\frac{1}{2}}) - \PP(I_{\frac{1}{2}})\right)$. Combining the two equalities yields $\PP(E_{\frac{1}{2}}) =\frac{1}{8} + \frac{1}{8} \int\limits_{\mathbb{R}^3} \mathrm{sgn}(xyz)\cdot f(x,y,z)\ dxdydz$.
As noted, no closed expression for $f$ is known, so the calculation of the above integral cannot be carried out in a straightforward manner. Instead, \eqref{Fourier isometry} allows us to rewrite the integral as
$$\int\limits_{\mathbb{R}^3} f(x,y,z)\cdot \mathrm{sgn}(xyz) \ dxdydz = \frac{1}{\pi^3}\mint{\times}\limits_{\mathbb{R}^3} \varphi(a,b,c)\cdot \widehat{\mathrm{sgn}}(abc) \ dadbdc,$$
where $\varphi$ is the characteristic function of $f$, and $\widehat{\mathrm{sgn}}$ is the Fourier transform of $\mathrm{sgn}_{(0,0,0)}$ as in \eqref{shifted sign}.\\

Thus, we are required to calculate $\varphi(a,b,c)$. Consider three independent normal random variables, $X,Y,Z$, with mean $0$ and variance $\sigma^2$, the characteristic function of $(XY,XZ,YZ)$ is defined by $(a,b,c)\to E[\exp(i(a\cdot XY + b\cdot XZ +c \cdot YZ))]$. We have that 
\[
a\cdot XY + b\cdot XZ +c \cdot YZ =\\ 
\mathrm{Tr}  \Bigg( \left[ \begin{array}{ccc}
0 &\frac{a}{2} & \frac{b}{2} \\
\frac{a}{2} & 0 & \frac{c}{2} \\
\frac{b}{2} & \frac{c}{2} & 0 \end{array} \right]  \cdot 
\left[ \begin{array}{ccc}
X^2 & XY & XZ \\
XY & Y^2 & YZ \\
XZ & YZ & Z^2 \end{array} \right]\Bigg). \]
If we consider the Wishart distribution $\mathcal{W}_3(\Sigma_\sigma, 1)$, where $\Sigma_\sigma$ is a $\sigma^2$ scalar matrix, we note that the above function equals the characteristic function of $\mathcal{W}_3(\Sigma_\sigma, 1)$ on the matrix $\left[ \begin{array}{ccc}
0 &\frac{a}{2} & \frac{b}{2} \\
\frac{a}{2} & 0 & \frac{c}{2} \\
\frac{b}{2} & \frac{c}{2} & 0 \end{array} \right]. $
Using the formula \eqref{char of wishart}, this equals $\det\Bigg(\left[ \begin{array}{ccc}
1 &-\mathrm{\bf{i}}\sigma^2 a & -\mathrm{\bf{i}}\sigma^2 b \\
-\mathrm{\bf{i}}\sigma^2 a& 1 & -\mathrm{\bf{i}}\sigma^2 c \\
-\mathrm{\bf{i}}\sigma^2 b & -\mathrm{\bf{i}}\sigma^2 c& 1 \end{array} \right] \Bigg)^{-\frac{1}{2}}$,
which may be written otherwise as $(1 + (\sigma^2)^2(a^2 + b^2 + c^2) + 2(\sigma^2)^3abc\mathrm{\bf{i}})^{-\frac{1}{2}}$.\\
\\
By the convolution-multiplication theorem ~\cite[Theorem 3.3.2]{durrett2010probability}, the characteristic function of a sum of independent variables is the multiplication of their characteristic functions, it then follows that: \begin{equation}\label{define phi}
\varphi(a,b,c) = \prod\limits_{i=1}^{d}(1 + \alpha_i^2(a^2+b^2+c^2) +2\alpha_i^3 abc\mathrm{\bf{i}})^{-\frac{1}{2}},
\end{equation}
which results in:
\begin{equation*}
\mint{\times}\limits_{\mathbb{R}^3}  \varphi(a,b,c)\cdot\widehat{\mathrm{sgn}}(abc) \ dadbdc= \mint{\times}\limits_{\mathbb{R}^3}\frac{\mathrm{\bf{i}}}{abc}\prod\limits_i(1 +\alpha_i^2(a^2 +b^2+c^2) + 2\alpha_i^3abc\mathrm{\bf{i}})^{-\frac{1}{2}}\ dadbdc.
\end{equation*}
This concludes the proof.
\end{proof}

In view of the above, it suffices to estimate the integral in \eqref{triangle prob}. We will show that the integral of the expression
$$\mathrm{Re}\left(\frac{\mathrm{\bf{i}}}{abc}\prod\limits_i(1 +\alpha_i^2(a^2 +b^2+c^2) + 2\alpha_i^3abc\mathrm{\bf{i}})^{-\frac{1}{2}}\right),$$
is concentrated in a ball of radius $\frac{1}{\norm{\alpha}_2}$, and that inside this ball, the above expression is very close in value to $\norm{\alpha}_3^3$. From this, it will follow that
$$\PP\left(E_{\frac{1}{2}}\right) \simeq \frac{1}{8} +  \left(\frac{\norm{\alpha}_3}{\norm{\alpha}_2}\right)^3.$$
The next result will be used to control the integral outside of the aforementioned ball.
\begin{lemma} \label{lem: integral bound}
Let $n \geq 3$ and $\gamma = \{\gamma_i\}_{i=1}^d$, suppose that $\gamma_i \in [0,1]$ for $1\leq i \leq d$. Define
$$I(T) = \int\limits_T^\infty \frac{r^2\ dr}{\sqrt{\prod\limits_i\left(1+\gamma_i^2r^2\right)}}, \ \ \forall T \geq 1,$$
and denote $\norm{\gamma}_2^2 = \sum\limits_i\gamma_i^2$, then there exist constants $c_n,C_n>0$, depending only on $n$, such that whenever $\norm{\gamma}_2^2>c_n$ we have that $I(T) \leq C_n\left(\frac{1}{\norm{\gamma}_2^{2}}\right)^{\frac{n}{2}}\frac{1}{T^{n-3}}$.
\end{lemma}
\begin{proof}
Indeed, assume $\norm{\gamma}_2^2 > n$. Note that necessarily $d \geq n$ in this case. Thus we can give a non trivial lower bound of $\prod\limits_i\left(1 +\gamma_i^2r^2\right)$ by considering the sum of all products of $n$ different elements of $\gamma$. That is
$$\prod\limits_i\left(1 +\gamma_i^2r^2\right) \geq\left(\sum\limits_{\substack{S \subset \gamma \\ |S| = n}}\prod\limits_{\gamma_j\in S}\gamma_j^2\right)r^{2n}.$$
We claim now that: 
\begin{equation} \label{bound}
\sum\limits_{\substack{S \subset \gamma \\ |S| = n}}\prod\limits_{\gamma_j\in S}\gamma_j^2 \geq \frac{1}{n!}\prod_{k=0}^{n-1}\left(\norm{\gamma}_2^2-k\right).
\end{equation}
To see that, we may rewrite 
$$\sum\limits_{\substack{S \subset \gamma \\ |S| = n}}\prod\limits_{\gamma_i\in S}\gamma_i^2 =\frac{1}{n} \sum\limits_i \gamma_i^2\sum\limits_{\substack{S \subset \gamma \setminus\{\gamma_i\} \\ |S| = n-1}}\prod\limits_{\gamma_j\in S}\gamma_j^2,$$
where we have counted each $S \subset \gamma$, $n$ times. But, $\gamma_i \leq 1$ for every $1 \leq i \leq d$, and so $\norm{\gamma \setminus \{\gamma_i\}}_2^2 > \norm{\gamma}_2^2 - 1$. \eqref{bound} now follows by induction, since
$$\frac{1}{n}\sum\limits_i \gamma_i^2\sum\limits_{\substack{S \subset \gamma \setminus\{\gamma_i\} \\ |S| = n-1}}\prod\limits_{\gamma_j\in S}\gamma_j^2\geq\frac{1}{n} \sum\limits_i \gamma_i^2\frac{1}{(n-1)!} \prod\limits_{k=0}^{n-2}(\norm{\gamma}_2^2-1-k) = \frac{1}{n!}\prod_{k=0}^{n-1}\left(\norm{\gamma}_2^2-k\right)$$
If we further assume that $\norm{\gamma}_2^2 \geq 2n$, then $\norm{\gamma}_2^2 - k > \frac{1}{2}\norm{\gamma}_2^2$, for every $0 \leq k \leq n-1$. Plugging this into \eqref{bound} produces
$$ \prod\limits_i\left(1 +\gamma_i^2r^2\right) \geq \left(\frac{\norm{\gamma}_2^2}{n!2}\right)^{n}r^{2n},$$
which implies
$$I(T) \leq \left(\frac{n!2}{\norm{\gamma}_2^{2}}\right)^{\frac{n}{2}}\int\limits_{T}^\infty\frac{dr}{r^{n-2}} = \frac{(n!2)^n}{n-3} \left(\frac{1}{\norm{\gamma}_2^2}\right)^{\frac{n}{2}}\frac{1}{T^{n-3}},$$ 
as desired.\\
\end{proof}
{\it Remark:} The constants obtained in the above proof are far from optimal, but will suffice for our needs.

We will use the above result in order to bound from below the integral in formula \eqref{triangle prob}.  For this, we will assume W.L.O.G. that $\alpha$ is normalized in the following way: \begin{equation}\label{normalization}
\alpha_1 = 1 \mbox{ and } \alpha_i \in [0,1] \mbox{ for }  1 \leq i \leq d.
\end{equation} 
We note that this normalization yields the following properties for $n,m \in \NN$, which we shall use freely:
\begin{itemize}
\item For every $k>0$, $\norm{\alpha}_k^k \geq 1$ and thus $\left(\norm{\alpha}_k^k\right)^n \leq \left(\norm{\alpha}_k^k\right)^m$ when $n \leq m$.
\item $\alpha_i^n \geq \alpha_i^m$ and $\norm{\alpha}_n^n \geq \norm{\alpha}_m^m$ when $n \leq m$.
\item For any $n>2$ and $\eps>0$ there exists $c>0$ such that whenever $\norm{\alpha}_2^2 > c$ we have $\left(\frac{\norm{\alpha}_n}{\norm{\alpha}_2}\right)^n<\eps$.
\end{itemize}
\begin{lemma} \label{lem: lower bound half}
There exists a constant $c_{1/2} > 0$ such that whenever $\norm{\alpha}_2^2>c_{1/2}$ then $$ \mint{\times}\limits_{\mathbb{R}^3}\frac{\mathrm{\bf{i}}}{abc}\prod\limits_i\left(1 +\alpha_i^2(a^2 +b^2+c^2) + 2\alpha_i^3abc\mathrm{\bf{i}}\right)^{-\frac{1}{2}}\ dadbdc \geq \frac{1}{100}\left(\frac{\norm{\alpha}_3}{\norm{\alpha}_2}\right)^3.$$
\end{lemma}
\begin{proof}
First, we have the privilege of knowing the integral evaluates to some probability. Therefore, the principal value of it's imaginary part must vanish. This becomes evident by noting that the imaginary part is an odd function. Thus, we are interested in:
\begin{align} \label{big integral}
\nonumber \mathrm{Re}&\left(\mint{\times}\limits_{\mathbb{R}^3}\frac{\mathrm{\bf{i}}}{abc}\prod\limits_i(1 +\alpha_i^2(a^2 +b^2+c^2) + 2\alpha_i^3abc\mathrm{\bf{i}})^{-\frac{1}{2}}\ dadbdc \right)\\ \nonumber
=&\mint{\times}\limits_{\mathbb{R}^3}\frac{-1}{abc}\mathrm{Im}\Bigg( \prod\limits_i(1 +\alpha_i^2(a^2 +b^2+c^2) + 2\alpha_i^3abc\mathrm{\bf{i}})^{-\frac{1}{2}} \Bigg)\ dadbdc \\ \nonumber
=&\mint{\times}\limits_{\mathbb{R}^3} \frac{-\sin\left(\arg\left(\prod\limits_i\left( 1+\alpha_i^2(a^2+b^2+c^2) + 2\alpha_i^3abc\mathrm{\bf{i}}\right)^{-\frac{1}{2}}\right)\right)}{abc \left | \prod\limits_i\Big( 1+\alpha_i^2(a^2+b^2+c^2) +2\alpha_i^3abc\mathrm{\bf{i}}\Big)^{\frac{1}{2}} \right |}dadbdc\\
=& \mint{\times}\limits_{\mathbb{R}^3} \frac{\sin\Big(\frac{1}{2}\sum\limits_i\arctan\left(\frac{2\alpha_i^3abc}
{1+\alpha_\mathrm{\bf{i}}^2\left(a^2+b^2+c^2\right)}\right)\Big)}{abc\prod\limits_i\left(\left(1+\alpha_i^2(a^2+b^2+c^2)\right)^2 +
4\alpha_i^6a^2b^2c^2\right)^{\frac{1}{4}}}
dadbdc \nonumber \\
=&\mint{\times}\limits_{\mathbb{R}^3}\frac{-\mathrm{Im}(\vphi(a,b,c))}{abc}\ dadbdc=\int\limits_{\mathbb{R}^3}\frac{-\mathrm{Im}(\vphi(a,b,c))}{abc}\ dadbdc,
\end{align}
where $\vphi$ is as in \eqref{define phi}. It is straightforward to verify that $\mathrm{Im}(\vphi(a,b,c) ) = O(abc)$, which implies that the above integrand is actually integrable, and thus justifies the last equality. We will estimate the above integral in several steps.
\subsubsection*{Step 1 - The integral is bounded from below on $B_1 = \Big\{x \in \RR^3 : \norm{x}^2 \leq \frac{1}{\norm{\alpha}^2_2} \Big\}$, the ball of radius $\frac{1}{\norm{\alpha}_2}$.}
First, we will prove that the following holds:
\begin{equation} \label{bound-arg}
\sin\left(\frac{1}{2}\sum\limits_i\arctan\left(\frac{2\alpha_i^3abc}{1+\alpha_i^2(a^2+b^2+c^2)}\right)\right) \geq \sum\limits_i\frac{\alpha_i^3abc}{1+\alpha_i^2(a^2+b^2+c^2)} - 3\norm{\alpha}_3^6(abc)^2.
\end{equation} 
Indeed, since $\sin(x) \geq x - x^2$ we have that

\begin{align*}
&\sin\left(\frac{1}{2}\sum\limits_i\arctan\left(\frac{2\alpha_i^3abc}{1+\alpha_i^2(a^2+b^2+c^2)}\right)\right)\\
\geq& \frac{1}{2}\sum\limits_i\arctan\left(\frac{2\alpha_i^3abc}{1+\alpha_i^2(a^2+b^2+c^2)}\right) - \frac{1}{4}\left(\sum\limits_i\arctan\left(\frac{2\alpha_i^3abc}{1+\alpha_i^2(a^2+b^2+c^2)}\right)\right)^2 \\
\geq&\frac{1}{2}\sum\limits_i\arctan\left(\frac{2\alpha_i^3abc}{1+\alpha_i^2(a^2+b^2+c^2)}\right) - \left(\sum\limits_i\alpha_i^3\right)^2(abc)^2.
\end{align*}
With the last inequality following from the fact that $\arctan^2(x) \leq x^2$.
Now, using the inequality $\arctan(x) \geq x - x^2$ yields
\begin{align*}
&\frac{1}{2}\sum\limits_i\arctan\left(\frac{2\alpha_i^3abc}{1+\alpha_i^2(a^2+b^2+c^2)}\right) - \left(\sum\limits_i\alpha_i^3\right)^2(abc)^2\\
\geq&\sum\limits_i\frac{\alpha_i^3abc}{1+\alpha_i^2(a^2+b^2+c^2)} - 2\left(\sum\limits_i\alpha_i^6\right)(abc)^2 - \left(\sum\limits_i\alpha_i^3\right)^2(abc)^2\\
\geq& \sum\limits_i\frac{\alpha_i^3abc}{1+\alpha_i^2(a^2+b^2+c^2)} - 3\norm{\alpha}_3^6(abc)^2.
\end{align*}
When $(a,b,c)\in B_1$, then $\alpha_i^2(a^2+b^2+c^2) \leq \frac{\alpha_i^2}{\norm{\alpha}_2^2} \leq 1$ and we have
\begin{equation} \label {bound arg 2}
\sum\limits_i\frac{\alpha_i^3abc}{1+\alpha_i^2(a^2+b^2+c^2)}- 3\norm{\alpha}_3^6(abc)^2 \geq \frac{1}{2}\norm{\alpha}_3^3abc - 3\norm{\alpha}_3^6(abc)^2.
\end{equation}
\\
Next, we note that for $(a,b,c) \in B_1$: $$1 \geq \frac{1}{\prod\limits_i\Big[\big(1+\alpha_i^2(a^2+b^2+c^2)\big)^2 +
4\alpha_i^6a^2b^2c^2\Big]^{\frac{1}{4}}} \geq \frac{1}{\prod\limits_i\Big[\left(1+\frac{\alpha_i^2}{\norm{\alpha}^2_2}\right)^2 + \frac{4\alpha_i^6}{\norm{\alpha}_2^6}\Big]^{\frac{1}{4}}}.$$ 
Since, in \eqref{normalization}, we've assumed that $\alpha_i \leq 1$ for each $i$ while $\sum\limits_{i} \alpha_i^2 \geq 1$, we may now lower bound the above by $\frac{1}{\prod\limits_i\left(1+\frac{7\alpha_i^2}{\norm{\alpha}_2^2}\right)^{-\frac{1}{4}}}$, and since $\ln\left(\prod\limits_i\left(1+\frac{7\alpha_i^2}{\norm{\alpha}_2^2}\right)\right)\leq\frac{7}{\norm{\alpha}_2^2}\sum\limits_i\alpha_i^2=7$, we have
\begin{equation} \label{bound-norm}
\frac{1}{\prod\limits_i\left(1+\frac{7\alpha_i^2}{\norm{\alpha}_2^2}\right)^{\frac{1}{4}}} \geq e^{-2}.\end{equation}
By combining \eqref{bound arg 2} and \eqref{bound-norm} into \eqref{define phi} we may see for $(a,b,c) \in B_1$ the following holds: $$ \text{Im}\left(\varphi(a,b,c)\right) \geq \left(\frac{1}{2}\norm{\alpha}_3^3abc - 3\norm{\alpha}_3^6\left(abc\right)^2\right)e^{-2} \mbox{ when } abc > 0.$$
Also, it is not hard to see that $\mathrm{Im}(\vphi)$ is an odd function, which makes $\frac{\text{Im}(\varphi(a,b,c))}{abc}$ even. Hence, if $H =\left\{(a,b,c)\in\RR^3|abc>0\right\}$, then
$$\int\limits_{B_1} \frac{\text{Im}(\varphi(a,b,c))}{abc} dadbdc= 2\int\limits_{B_1\cap H} \frac{\text{Im}(\varphi(a,b,c))}{abc} dadbdc.$$
Finally, since the volume of $B_1$ is $\frac{4\pi}{3\norm{\alpha}_2^3}$, and as long as $\norm{\alpha}^2_2$ is large enough:
\begin{align*}
&\int\limits_{B_1\cap H} \frac{-\text{Im}(\varphi(a,b,c))}{abc}dadbdc\geq \frac{1}{e^2}\int\limits_{B_1 \cap H}\left(\frac{1}{2}\norm{\alpha}_3^3 - 3\norm{\alpha}_3^6abc\right)dadbdc\\ 
\geq& \frac{\pi}{3e^2}\left(\frac{\norm{\alpha}_3}{\norm{\alpha}_2}\right)^3 - \frac{3\norm{\alpha}_3^6}{e^2}\int\limits_{B_1}|abc|\ dadbdc\geq \frac{\pi}{3e^2}\left(\frac{\norm{\alpha}_3}{\norm{\alpha}_2}\right)^3 - \frac{3}{e^2}\left(\frac{\norm{\alpha}_3}{\norm{\alpha}_2}\right)^6,
\end{align*}
where the last inequality uses the fact
$$\int\limits_{B_1}|abc|\ dadbdc \leq \frac{1}{\norm{\alpha}_2^6}.$$
Now, by using the properties of the normalization \eqref{normalization}, $\norm{\alpha}^3_3 \leq \norm{\alpha}^2_2$. Thus,
$$\left(\frac{\norm{\alpha}_3}{\norm{\alpha}_2}\right)^6 \leq \frac{1}{\norm{\alpha}_2}\left(\frac{\norm{\alpha}_3}{\norm{\alpha}_2}\right)^3,$$ and there exists a constant $c_1> 0$ such that whenever $\norm{\alpha}_2^2 > c_1$ then 
$$\int\limits_{B_1} \frac{-\text{Im}(\varphi(a,b,c))}{abc} dadbdc>\frac{\pi}{4e^2} \left(\frac{\norm{\alpha}_3}{\norm{\alpha}_2}\right)^3 > \frac{1}{10} \left(\frac{\norm{\alpha}_3}{\norm{\alpha}_2}\right)^3 .$$ 

\subsubsection*{Step 2 - The integrand is positive on $B_2 = \left\{x \in \RR^3 : \norm{x}^2 \leq \frac{1}{\norm{\alpha}_2^{22/12}} \right\}$, the ball of radius $\frac{1}{\norm{\alpha}_2^{11/12}}$.}

We first note that whenever $\left | \sum\limits_i\arctan\left(\frac{2\alpha_i^3abc}
 {1+\alpha_i^2(a^2+b^2+c^2)}\right) \right | < \pi$, then
 $$\sin\left(\arg\left(\prod\limits_i\left[ 1+\alpha_i^2(a^2+b^2+c^2) +2\alpha_i^3abc\mathrm{\bf{i}}\right]\right)\right)$$ has the same sign as that of $abc$, which in turn implies that $\frac{-\mathrm{Im}(\vphi(a,b,c))}{abc} > 0$. Thus, it will suffice to show that whenever $(a,b,c) \in B_2$ and $abc>0$, we have $\sum\limits_i\arctan\left(\frac{2\alpha_i^3abc}
  {1+\alpha_i^2(a^2+b^2+c^2)}\right) < \pi$.\\
\\
Indeed, for $(a,b,c) \in B_2$, $abc < \left(\norm{\alpha}_2^{-11/12}\right)^3 \leq \frac{1}{\norm{\alpha}_2^2}$ which, under the assumption $abc>0$, results in 
$$\sum\limits_i\arctan\left(\frac{2\alpha_i^3abc}
  {1+\alpha_i^2(a^2+b^2+c^2)}\right) \leq \sum\limits_i\frac{2\alpha_i^3abc}
    {1+\alpha_i^2(a^2+b^2+c^2)} \leq  \frac{2\norm{\alpha}_3^3}{\norm{\alpha}_2^2} < 2 <\pi,$$ as desired.

\subsubsection*{Step 3 - The absolute value of the integrand is negligible on the spherical shell $B \setminus B_2$ where $B$ is the unit ball in $\RR^3$.}

Observe that,   
\begin{equation} \label{norm3 bound}
\left|\frac{\sin\left(\frac{1}{2}\sum\limits_i\arctan\left(\frac{2\alpha^3abc}
{1+\alpha_i^2(a^2+b^2+c^2)}\right)\right)}
{abc}\right| \leq \frac{1}{2}\sum\limits_i \frac{\frac{2\alpha_i^3|abc|}{1+\alpha_i^2(a^2+b^2+c^2)}}{|abc|} \leq \norm{\alpha}_3^3.
\end{equation} 
On the other hand, for $(a,b,c) \notin B_2$ we have that :
\begin{align*}
\frac{1}{\prod\limits_i\Big[\left(1+\alpha_i^2(a^2+b^2+c^2)\right)^2 + 4\alpha_i^6a^2b^2c^2\Big]^{\frac{1}{4}}} \leq&
\frac{1}{\prod\limits_i(1+\alpha_i^2(a^2+b^2+c^2))^{\frac{1}{2}}} \\\leq&\prod\limits_i\left(1+\frac{\alpha_i^2}{\norm{\alpha}_2^{22/12}}\right)^{-\frac{1}{2}}.
\end{align*}
Using the elementary inequality $\ln(1+x) \geq x - \frac{x^2}{2}$ for $x>0$ yields: $$\ln\left(\prod\limits_i\left(1+\frac{\alpha_i^2}{\norm{\alpha}_2^{22/12}}\right)\right) = \sum\limits_i\ln\left(1+\frac{\alpha_i^2}{\norm{\alpha}_2^{22/12}}\right) \geq \norm{\alpha}_2^{2/12} - \frac{ \norm{\alpha}_4^4}{2\norm{\alpha}_2^{44/12}} \geq \norm{\alpha}_2^{2/12} -1$$
where the last inequality follows from the fact that $\norm{\alpha}_4^4 \leq \norm{\alpha}_2^2$. In turn, this implies $$\prod\limits_i\left(1+\frac{\alpha_i^2}{\norm{\alpha}_2^{22/12}}\right)^{-\frac{1}{2}} \leq e^{-\frac{\norm{\alpha}_2^{2/12}-1}{2}}.$$\\
Finally, since the volume of the unit ball is $\frac{4\pi}{3}$, this gives
\begin{equation} \label{step 2}
\int\limits_{B\setminus B_2} \left|\frac{\text{Im}(\varphi(a,b,c))}{abc}\right|dadbdc <\frac{4\pi}{3} \norm{\alpha}_3^3e^{-\frac{\norm{\alpha}^{2/12}_2-1}{2}}.
\end{equation} Consequently, there is a constant $c_2$ such that whenever $\norm{\alpha}_2^2 > c_2$ then 
\begin{equation*} 
\int\limits_{B\setminus B_2} \left|\frac{\text{Im}(\varphi(a,b,c))}{abc}\right|dadbdc\leq\frac{1}{100}\left(\frac{\norm{\alpha}_3}{\norm{\alpha}_2}\right)^3.
\end{equation*}

\subsubsection*{Step 4 - The integral is negligible outside of $B$.}

For $(a,b,c) \notin B$ we use \eqref{norm3 bound} to achieve $$\frac{\sin\Big(\frac{1}{2}\sum\limits_i\arctan\left(\frac{2\alpha_i^3abc}
{1+\alpha_i^2\left(a^2+b^2+c^2\right)}\right)\Big)}{abc\prod\limits_i\left(\left(1+\alpha_i^2(a^2+b^2+c^2)\right)^2 +
4\alpha_i^6a^2b^2c^2\right)^{\frac{1}{4}}} < \frac{\norm{\alpha}_3^3}{\prod\limits_i\left(1+\alpha_i^2(a^2+b^2+c^2)\right)^{\frac{1}{2}}}.$$
By passing to spherical coordinates we obtain: 
\begin{align*}
\int\limits_{\RR^3 \setminus B} \frac{1}{\prod\limits_i\left(1+\alpha_i^2(a^2+b^2+c^2)\right)^{\frac{1}{2}}}dadbdc= 
4\pi \int\limits_1^\infty \frac{r^2\ dr}{\prod\limits_i(1+\alpha_i^2r^2)^{\frac{1}{2}}}.
\end{align*}
Applying Lemma \ref{lem: integral bound} with $n=4$ and $T=1$, shows the existence of constants $C,c_3'>0$ such that whenever $\norm{\alpha}_2^2>c_3'$,
\begin{equation*} 
\int\limits_1^\infty \frac{r^2\ dr}{\prod\limits_i(1+\alpha_i^2r^2)^{\frac{1}{2}}} \leq C\left(\frac{1}{\norm{\alpha}_2^2}\right)^2 = C\frac{1}{\norm{\alpha}_2^4}.
\end{equation*}
Thus, there exists a constant $c_3=\max(c_3',(16C)^2)$ such that whenever $\norm{\alpha}_2^2 > c_3$ then 
\begin{equation*}
\int\limits_{\RR^3 \setminus B}\left|\frac{\text{Im}(\varphi(a,b,c))}{abc}\right|dadbdc \leq\frac{1}{100}\left(\frac{\norm{\alpha}_3}{\norm{\alpha}_2}\right)^3.
\end{equation*}

\subsubsection*{Final Step - $\int\limits_{\RR^3} \frac{-\text{Im}(\varphi(a,b,c))}{abc}dadbdc \geq \frac{1}{100}\left(\frac{\norm{\alpha}_3}{\norm{\alpha}_2}\right)^3$}

We may now decompose the integral 
$$\int\limits_{\RR^3} \frac{-\text{Im}(\varphi(a,b,c))}{abc} dadbdc= \int\limits_{B_2} \frac{-\text{Im}(\varphi(a,b,c))}{abc}dadbdc + \int\limits_{\RR^3 \setminus B_2} \frac{-\text{Im}(\varphi(a,b,c))}{abc}dadbdc$$
Letting $\norm{\alpha}_2^2 > \max(c_1,c_2,c_3)$ steps 1 and 2 show that 
\begin{equation}\label{big in ball}
\int\limits_{B_2} \frac{-\text{Im}(\varphi(a,b,c))}{abc}dadbdc \geq \frac{1}{10}\left(\frac{\norm{\alpha}_3}{\norm{\alpha}_2}\right)^3,
\end{equation}
while steps 2 and 3 show
\begin{equation*} \label{small not in ball}
\int\limits_{\RR^3 \setminus B_2}\left| \frac{\text{Im}(\varphi(a,b,c))}{abc}\right|dadbdc \leq \frac{2}{100}\left(\frac{\norm{\alpha}_3}{\norm{\alpha}_2}\right)^3.
\end{equation*}
The required bound then follows by combining the above two estimates.

\end{proof}
\subsection{Arbitrary $0 < p < 1$}
We now consider the case for arbitrary $p$. First, we would like to derive bounds on the behavior of $t_{p,\alpha}$, which constitute the following lemma.
\begin{lemma} \label {lem: tp}
Let $p \in (0,1)$ and denote by $\Phi$ the cumulative distribution function of the standard Gaussian. If $t_p = \Phi^{-1}(p)$ then $\norm{\alpha}_2t_p - k_p\leq t_{p,\alpha} \leq \norm{\alpha}_2t_p + k_p$, for a constant $k_p$ depending only on $p$. Furthermore, if $p':=\Phi\left(\frac{t_{p,\alpha}}{\norm{\alpha}_2}\right)$ then $|p-p'| \leq 3\left(\frac{\norm{\alpha}_3}{\norm{\alpha}_2}\right)^3$.
\end{lemma}
\begin{proof}
Let $W = \frac{\langle X_1 , X_2 \rangle}{\norm{\alpha}_2}$ where $X_1,X_2$ are defined as in the beginning of the section. We may consider $\langle X_1, X_2 \rangle$ as sum of independent random variables $X_1^i\cdot X_2^i$, where for each $1\leq i \leq d$, $X_1^i$ and $X_2^i$ are independently distributed as $\mathcal{N}(0,\alpha_i)$. It then holds that $\EE[X_1^i\cdot X_2^i] = 0$, $\EE[(X_1^i\cdot X_2^i)^2] = \alpha_i^2$. The absolute third moments are given as a product of absolute third moments of Gaussians. That is, $\EE[|X_1^i\cdot X_2^i|^3] = \frac{8\alpha_i^3}{\pi} < 3\alpha_i^3$.\\
\\
Let $t$ be such that $p = \PP(W \geq t)$, in which case we also have $t_{p,\alpha} = t\norm{\alpha}_2$. Note that $$\frac{\sum\limits_i\EE[|X_1^i\cdot X_2^i|^3]}{\left(\sum\limits_i\EE[(X_1^i\cdot X_2^i)^2]\right)^{3/2}} \leq \frac{3\norm{\alpha}_3^3}{\norm{\alpha}_2^3}.$$ Thus, if $Z$ is a standard normal random variable, Berry-Esseen's inequality, \eqref{prelim berry eseen}, yields for every $s \in \RR$: 
\begin{equation*}
|\PP(W > s) - \PP(Z > s)| \leq \frac{3\norm{\alpha}_3^3}{\norm{\alpha}_2^3}.
\end{equation*}
If $t_p = \Phi^{-1}(p)$ then $\PP(Z > t_p) = p$ and 
$$|\Phi(t_p)-\Phi(t)|=|\PP(Z > t_p) -\PP(Z > t)|=|\PP(W > t) -\PP(Z > t)|\leq \frac{3\norm{\alpha}_3^3}{\norm{\alpha}_2^3}.$$
Since $|p - p'| = |\Phi(t_p)-\Phi(t)|$, this shows the second part of the statement. To finish the proof, denote $m= \inf\limits_{s\in[t_p,t]}(\Phi'(s))$. By Lagrange's theorem
$$m|t_p - t|\leq |\Phi(t_p)-\Phi(t)|\leq \frac{3\norm{\alpha}_3^3}{\norm{\alpha}_2^3} \leq \frac{3}{\norm{\alpha}_2},$$
which shows $t_{p,\alpha}\in \norm{\alpha}_2t_p \pm \frac{3}{m}$.
\end{proof}
Before proceeding, we need some further definitions. Let $X'_1,X'_2,X'_3$ be independent copies of $X_1,X_2,X_3$ and consider the joint distribution $(\langle X_1,X_2\rangle,\langle X_1',X_3\rangle,\langle X_2',X_3'\rangle)$. This distribution has independent coordinates. Denote its density by $g$ and corresponding characteristic function by $\psi$.
If $N_1,N_2$ are two independent standard Gaussians then the characteristic function of their product can be derived from \eqref{char of wishart} as $\EE e^{\mathrm{\bf{i}}tN_1N_2} = \left(1+t^2\right)^{-\frac{1}{2}}$. From this, it follows that the characteristic function of $\langle X_1,X_2\rangle$ is $\EE e^{\mathrm{\bf{i}}t\langle X_1,X_2\rangle} = \prod\limits_i\left(1+\alpha_i^2t^2\right)^{-\frac{1}{2}}$
, and we have, by independence
\begin{equation} \label{define psi}
\psi(a,b,c) = \prod\limits_i\left((1+\alpha_i^2a^2)(1+\alpha_i^2b^2)(1+\alpha_i^2c^2)\right)^{-\frac{1}{2}}.
\end{equation}
We denote by $\psi_1\left(a',b',c'\right) = \psi\left(\frac{a'}{\norm{\alpha}_2},\frac{b'}{\norm{\alpha}_2},\frac{c'}{\norm{\alpha}_2}\right)$ and $\varphi_1\left(a',b',c'\right) = \vphi\left(\frac{a'}{\norm{\alpha}_2},\frac{b'}{\norm{\alpha}_2},\frac{c'}{\norm{\alpha}_2}\right)$ for the characteristic function $\vphi$, \eqref{define phi}.
The following result will help us relate the independent version of the distribution and the original one.
\begin{lemma} \label{lem: coordinate free version}
There exist absolute constants $c,C,\eps>0$ such that whenever $\norm{\alpha}_2^2>c$ then $$\int\limits_{\RR^3}|\mathrm{Re}(\varphi_1) - \psi_1|da'db'dc' \leq C \left(\frac{\norm{\alpha}_3}{\norm{\alpha}_2}\right)^{3+\eps}.$$
\end{lemma}
\begin{proof}
Note that since $\psi_1$ and $\varphi_1$ are characteristic functions, then $|\psi_1|,|\re(\varphi_1)| \leq 1$. So, $|\psi_1 - \re(\varphi_1)| \leq |\ln(\psi_1)-\ln(\re(\varphi_1))|$. Now, let 
$$B_{0.01} = \left \{x \in \RR^3 : ||x||^2 \leq \left(\frac{\norm{\alpha}_2}{\norm{\alpha}_3}\right)^{0.01} \right \}.$$
 Clearly, $|\re(\varphi_1)| \leq |\varphi_1| = \prod\limits_i\left((1+\frac{\alpha_i^2}{\norm{\alpha}_2^2}(a'^2+b'^2+c'^2))^2 +
 4\frac{\alpha_i^6}{\norm{\alpha}_2^6}a'^2b'^2c'^2\right)^{-\frac{1}{4}}$, and since
 \begin{equation*}
 |a'b'c'| \leq \left(a'^2+b'^2+c'^2\right)^{\frac{3}{2}}\leq\left(\frac{\norm{\alpha}_2}{\norm{\alpha}_3}\right)^{0.015} \mbox{ for } (a',b',c') \in B_{0.01},
 \end{equation*}
 we have
 $$\left|\arg\left(1+\frac{\alpha_i^2}{\norm{\alpha}_2^2}(a'^2+b'^2+c'^2)+2\frac{\alpha_i^3}{\norm{\alpha}_2^3}a'b'c'\mathrm{\bf{i}}\right)\right| \leq 2\frac{\alpha_i^3}{\norm{\alpha}_2^3}|a'b'c'|\leq 2\frac{\alpha_i^3}{\norm{\alpha}_2^3}\left(\frac{\norm{\alpha}_2}{\norm{\alpha}_3}\right)^{0.015}.$$ By using the inequality $\cos(x) \geq 1 - x^2$, we achieve
 \begin{align*}
 \re(\varphi_1) \geq \cos\left(2\frac{\norm{\alpha}_3^3}{\norm{\alpha}_2^3}\left(\frac{\norm{\alpha}_2}{\norm{\alpha}_3}\right)^{0.015}\right)|\vphi_1|\geq \left(1 - 4\frac{\norm{\alpha}_3^6}{\norm{\alpha}_2^6}\left(\frac{\norm{\alpha}_2}{\norm{\alpha}_3}\right)^{0.03}\right)\left|\varphi_1\right|.
  \end{align*}
Using the above, together with the triangle inequality gives
\begin{equation} \label{in ball}
\left|\ln(\psi_1)-\ln(\re(\varphi_1))\right| \leq |\ln(\psi_1)-\ln(|\varphi_1|)| + \left|\ln\left(1 - 4\frac{\norm{\alpha}_3^6}{\norm{\alpha}_2^6}\left(\frac{\norm{\alpha}_2}{\norm{\alpha}_3}\right)^{0.03}\right)\right|.
\end{equation}
For $x\in (0,\frac{1}{2})$ we have the inequality $|\ln(1-x)|\leq 2x$, thus, as long as $\norm{\alpha}_2^2$ is large enough
\begin{equation*} 
\left|\ln\left(1 - 4\frac{\norm{\alpha}_3^6}{\norm{\alpha}_2^6}\left(\frac{\norm{\alpha}_2}{\norm{\alpha}_3}\right)^{0.03}\right)\right| \leq  8\frac{\norm{\alpha}_3^6}{\norm{\alpha}_2^6}\left(\frac{\norm{\alpha}_2}{\norm{\alpha}_3}\right)^{0.03},
\end{equation*}
and
\begin{equation} \label{ln-remainder}
8\int\limits_{B_{0.01}}\frac{\norm{\alpha}_3^6}{\norm{\alpha}_2^6}\left(\frac{\norm{\alpha}_2}{\norm{\alpha}_3}\right)^{0.03}da'db'dc' \leq 32\pi\frac{\norm{\alpha}_3^6}{\norm{\alpha}_2^6}\left(\frac{\norm{\alpha}_2}{\norm{\alpha}_3}\right)^{0.045} = 32\pi\left(\frac{\norm{\alpha}_3}{\norm{\alpha}_2}\right)^{5.955}.
\end{equation}
By using the inequality $|\ln(1+x)-x| \leq x^2$ for $x>0$ we bound $\ln(\psi_1)$ with
\begin{align*}
\ln(\psi_1(a',b',c'))=&-\frac{1}{2}\sum\limits_i\left[\ln\left(1+\frac{\alpha_i^2a'^2}{\norm{\alpha}_2^2}\right)+\ln\left(1+\frac{\alpha_i^2b'^2}{\norm{\alpha}_2^2}\right)+\ln\left(1+\frac{\alpha_i^2c'^2}{\norm{\alpha}_2^2}\right)\right]\\=& -\frac{1}{2}\left(a'^2+b'^2+c'^2\right) + O\left(\frac{\norm{\alpha}_4^4}{\norm{\alpha}_2^4}\right)\left(a'^4+b'^4+c'^4\right).
\end{align*}
Similar considerations show
\begin{align} \label{char taylor}
\ln(|\varphi_1|) =& -\frac{1}{4}\sum \ln \left(1 + \frac{2\alpha_i^2}{\norm{\alpha}_2^2}(a'^2+b'^2+c'^2)+\frac{\alpha_i^4}{\norm{\alpha}_2^4}(a'^2+b'^2+c'^2)^2+4\frac{\alpha_i^6}{\norm{\alpha}_2^{6}}a'^2b'^2c'^2\right)\nonumber\\=& -\frac{1}{2}\left(a'^2+b'^2+c'^2\right) - \frac{\norm{\alpha}_4^4}{4\norm{\alpha}_2^4}(a'^2+b'^2+c'^2)^2 - \frac{\norm{\alpha}_6^6}{\norm{\alpha}_2^{6}}a'^2b'^2c'^2\nonumber\\
+& O\left(\frac{\norm{\alpha}_4^4}{\norm{\alpha}_2^4}\right)\left((a'^2+b'^2+c'^2)^2 + \left(a'^2+b'^2+c'^2\right)^4 + a'^4b'^4c'^4\right)\nonumber\\
=&-\frac{1}{2}\left(a'^2+b'^2+c'^2\right) + O\left(\frac{\norm{\alpha}_4^4}{\norm{\alpha}_2^4}\right)\left(1 + \left(a'^2+b'^2+c'^2\right)^6\right).
\end{align}
The above shows the existence of a constant $C>0$ such that
\begin{align} \label{ln-in-ball}
\nonumber &\int\limits_{B_{0.01}}|\ln(\psi_1) - \ln(|\varphi_1|)| \leq C\left(\frac{\norm{\alpha}_4}{\norm{\alpha}_2}\right)^4\int\limits_{B_{0.01}}(a'^2+b'^2+c'^2)^6\ da'db'dc'\\ 
=& 4\pi C\left(\frac{\norm{\alpha}_4}{\norm{\alpha}_2}\right)^4\left(\frac{\norm{\alpha}_2}{\norm{\alpha}_3}\right)^{0.075}\leq 4\pi C\left(\frac{\norm{\alpha}_3}{\norm{\alpha}_2}\right)^4\left(\frac{\norm{\alpha}_2}{\norm{\alpha}_3}\right)^{0.075} = 4\pi C\left(\frac{\norm{\alpha}_3}{\norm{\alpha}_2}\right)^{3.925}.
\end{align}
By combining \eqref{ln-remainder},\eqref{ln-in-ball} and \eqref{in ball}, we obtain
$$\int\limits_{B_{0.01}}|\psi_1 - \re(\varphi_1)|da'db'dc' \leq \pi(4 C+32)\left(\frac{\norm{\alpha}_3}{\norm{\alpha}_2}\right)^{3.925}.$$
To bound the integral in $\RR^3 \setminus B_{0.01}$ we proceed in similar fashion to step 3 in Lemma \ref{lem: lower bound half}. First, note that
$$|\varphi_1|,|\psi_1| \leq \frac{1}{\prod\limits_i \left(1+\frac{\alpha_i^2}{\norm{\alpha}_2^2}(a'^2+b'^2+c'^2)\right)^{\frac{1}{2}}}.$$
Denoting $r = \sqrt{a'^2+b'^2+c'^2}, T = \left(\frac{\norm{\alpha}_2}{\norm{\alpha}_3}\right)^{0.005}$ and passing to spherical coordinates yields
$$\int\limits_{\RR^3 \setminus B_{0.01}}|\re(\varphi_1) - \psi_1|da'db'dc' \leq \int\limits_{\RR^3 \setminus B_{0.01}}|\re(\varphi_1)| + |\psi_1|da'db'dc' \leq  8\pi\int\limits_{T}^{\infty}\frac{r^2\ d r}{\prod\limits_i\left(1+\frac{\alpha_i^2}{\norm{\alpha}_2^2}r^2\right)^{\frac{1}{2}}}.$$
Invoking Lemma \ref{lem: integral bound} with $n>606$ shows the existence of constants $C,c>0$ such that
$$\int\limits_{T}^{\infty}\frac{r^2\ d r}{\prod\limits_i\left(1+\frac{\alpha_i^2}{\norm{\alpha}_2^2}r^2\right)^{\frac{1}{2}}} \leq CT^{-603} = C\left(\frac{\norm{\alpha}_3}{\norm{\alpha}_2}\right)^{3.015},$$
whenever $\norm{\alpha}_2^2>c$. This concludes the proof when we take $\eps = 0.015$.
\end{proof}
We are now ready to bound from below the probability of an induced triangle occurring in the general setting. Set $p \in (0,1)$ and $t := t_{p,\alpha}$. We are interested in the event $$\Big\{ \min\left(\langle X_1,X_2\rangle,\langle X_1,X_3\rangle,\langle X_2,X_3\rangle\right) > t\Big\}.$$
As before, let $f$ be the joint density of $(\langle X_1,X_2\rangle,\langle X_1,X_3\rangle,\langle X_2,X_3\rangle)$ and consider the integral:
$$I_p := \int\limits_{\RR^3} f(x,y,z) \mathrm{sgn}(x-t)\mathrm{sgn}(y-t)\mathrm{sgn}(z-t)\ dxdydz.$$
Note that, in the above formula, replacing $f$ with $g$, the density of the coordinate-independent version, as defined above, would yield $I_p = p^3 + 3(1-p)^2p - 3(1-p)p^2 - (1-p)^3 = (2p - 1)^3$.\\
\\
For the rest of this section, our goal will be to show that $I_p$ is large, compared to $(2p-1)^3$. That is, the dependency between the coordinates induces an increased probability for triangles and induced edges.
As in \eqref{invere char function}, we may write the Fourier transform of $\mathrm{sgn}(x-t)\mathrm{sgn}(y-t)\mathrm{sgn}(z-t)$ as $\widehat{\mathrm{sgn}}(a,b,c)e^{-2\pi\mathrm{\bf{i}}t(a+b+c)}$. Thus, by \eqref{Fourier isometry}, we have the equality
$$I_p = \frac{1}{\pi^3}\mint{\times}\limits_{\RR^3}\varphi(a,b,c)\widehat{\mathrm{sgn}}(a,b,c)e^{-2\pi \mathrm{\bf{i}}t(a+b+c)}\ dadbdc,$$   
where $\varphi$, as in \eqref{define phi}, is the characteristic function of $f$.
Since $I_p$ represents a real number, we only need to consider the real part of the integral:
\begin{align*}
I_p =&\frac{1}{\pi^3} \mint{\times} \mathrm{Re}\left(\varphi(a,b,c)\widehat{\mathrm{sgn}}(a,b,c)\right)\cos(2\pi t(a+b+c))dadbdc \\
+&\frac{1}{\pi^3}\mint{\times} \mathrm{Im} \left(\varphi(a,b,c)\widehat{\mathrm{sgn}}(a,b,c)\right)\sin(2\pi t(a+b+c))dadbdc\\
=&\frac{1}{\pi^3} \mint{\times}\frac{-\mathrm{Im}\left(\varphi(a,b,c)\right)}{abc}\cos(2\pi t(a+b+c))dadbdc\\
+&\frac{1}{\pi^3}\mint{\times} \frac{ \mathrm{Re}\left(\varphi(a,b,c)\right)}{abc}\sin(2\pi t(a+b+c))dadbdc.
\end{align*}
We denote $$I'_p = \frac{1}{\pi^3}\mint{\times}\frac{\mathrm{Re} \left(\varphi(a,b,c)\right)}{abc}\sin(2\pi t(a+b+c))dadbdc$$ and $$I''_p = \frac{1}{\pi^3}\mint{\times}\frac{-\mathrm{Im}\left(\varphi(a,b,c)\right)}{abc}\cos(2\pi t(a+b+c))dadbdc.$$\\
In the proof of Lemma \ref{lem: lower bound half} we have seen that
$$\mint{\times}\frac{-\mathrm{Im}\left(\varphi(a,b,c)\right)}{abc}dadbdc,$$
is mostly concentrated near the origin. Thus, we should expect that
$$\mint{\times}\frac{-\mathrm{Im}\left(\varphi(a,b,c)\right)}{abc}\cos(2\pi t(a+b+c))dadbdc \simeq \mint{\times}\frac{-\mathrm{Im}\left(\varphi(a,b,c)\right)}{abc}dadbdc.$$
So that $I''_p$ represents the increase in probability.
From Lemma \ref{lem: coordinate free version}, we know that $\mathrm{Re}(\vphi)$ is close to $\psi$, the characteristic function of the coordinate-independent version, which means that $I'_p$ should be close to $(2p-1)^3$. The next two claims will formalize this intuition. We begin by showing that $I''_p$ is large. 
\begin{claim} \label{I''}
	Fix $p \in (0,1)$. There exist constants $\delta'_p,c_p>0$ depending only on $p$ such that whenever $\norm{\alpha}_2^2>c_p$ then $I''_p  \geq 2\delta'_p\left(\frac{\norm{\alpha}_3}{\norm{\alpha}_2}\right)^3$. 
\end{claim}
\begin{proof}
	First, it is not hard to see that the integrand in $I''_p$ is continuous, up to a removable discontinuity, and we may pass to standard integration. Let $R$ be an arbitrary orthogonal transformation which takes $(1,0,0)$ to $\frac{1}{\sqrt{3}}(1,1,1)$. Consider the set
	$$K= R\left(\left[-\frac{1}{\norm{\alpha}^{11/12}_2},\frac{1}{\norm{\alpha}^{11/12}_2}\right]\times\left[-\frac{1}{\norm{\alpha}^{11/12}_2},\frac{1}{\norm{\alpha}^{11/12}_2}\right]\times\left[-\frac{1}{\norm{\alpha}^{11/12}_2},\frac{1}{\norm{\alpha}^{11/12}_2}\right]\right).$$
	Note that if $B_2=\Big\{x \in \RR^3 | \norm{x}^2\leq \frac{1}{\norm{\alpha}_2^{22/12}}\Big\}$ and $B_2'=\Big\{x \in \RR^3 | \norm{x}^2\leq \frac{4}{\norm{\alpha}_2^{22/12}}\Big\}$ then,
	$$B_2 \subset K \subset B_2'.$$
	Now, recall from \eqref{big integral} that, 
	$$\frac{-\mathrm{Im}(\vphi(a,b,c))}{abc} = \frac{\sin\Big(\frac{1}{2}\sum\limits_i\arctan\left(\frac{2\alpha_i^3abc}
		{1+\alpha_i^2\left(a^2+b^2+c^2\right)}\right)\Big)}{abc\prod\limits_i\left(\left(1+\alpha_i^2(a^2+b^2+c^2)\right)^2 +
		4\alpha_i^6a^2b^2c^2\right)^{\frac{1}{4}}}.$$
	From \eqref{norm3 bound} and \eqref{bound-arg}, we have
	$$\norm{\alpha}_3^3 \geq \frac{\sin\left(\frac{1}{2}\sum\limits_i\arctan\left(\frac{2\alpha_i^3abc}{1+\alpha_i^2\left(a^2+b^2+c^2\right)}\right)\right)}{abc} \geq \sum\limits_i\frac{\alpha_i^3}{1+\alpha_i^2(a^2+b^2+c^2)} - 3\norm{\alpha}_3^6|abc|.$$
	Along with the inequality $\frac{\alpha_i^3}{1+\alpha_i^2\left(a^2+b^2+c^2\right)}\geq\alpha_i^3\left(1-\alpha_i^2\left(a^2+b^2+c^2\right)\right)$, the above yields
	$$\left|\frac{\sin\Big(\frac{1}{2}\sum\limits_i\arctan\left(\frac{2\alpha_i^3abc}
		{1+\alpha_i^2\left(a^2+b^2+c^2\right)}\right)\Big)}{abc} - \norm{\alpha}_3^3\right| \leq \norm{\alpha}^5_5(a^2+b^2+c^2) - 3\norm{\alpha}_3^6|abc|.$$
	Therefore
	\begin{align} \label{first approx}
	&\int\limits_{K}\frac{-\mathrm{Im}\left(\varphi(a,b,c)\right)}{abc}\cos(2\pi t(a+b+c))dadbdc \nonumber\\ 
	\geq \norm{\alpha}_3^3&\int\limits_{K}\frac{\cos(2\pi t(a+b+c))dadbdc}{\prod\limits_i\left(\left(1+\alpha_i^2\left(a^2+b^2+c^2\right)\right)^2 +
		4\alpha_i^6a^2b^2c^2\right)^{\frac{1}{4}}} \nonumber \\
	- 3\norm{\alpha}_3^6&\int\limits_{K}|abc|dadbdc -\norm{\alpha}_5^5\int\limits_{K}(a^2+b^2+c^2)dadbdc,
	\end{align}
	with 
	\begin{align*}
	3&\norm{\alpha}_3^6\int\limits_{K}|abc|dadbdc\leq C_1\frac{\norm{\alpha}^6_3}{\norm{\alpha}^{5.5}_2} = C_1\left(\frac{\norm{\alpha}_3}{\norm{\alpha}_2}\right)^3 \frac {\norm{\alpha}_3^3} {\norm{\alpha}_2^2} \frac{1}{\norm{\alpha}_2^{0.5}} \leq C_1\left(\frac{\norm{\alpha}_3}{\norm{\alpha}_2}\right)^3\frac{1}{\norm{\alpha}_2^{0.5}} ,\\
	&\norm{\alpha}_5^5\int\limits_{K}(a^2+b^2+c^2)dadbdc \leq C_1\frac{\norm{\alpha}^5_5}{\norm{\alpha}^{55/12}_2}
	\leq C_1\left(\frac{\norm{\alpha}_3}{\norm{\alpha}_2}\right)^3\frac{1}{\norm{\alpha}_2},
	\end{align*}
	for an absolute constant $C_1>0$. Recalling that $$|\vphi(a,b,c)|=\prod\limits_i\left(\left(1+\alpha_i^2(a^2+b^2+c^2)\right)^2 +
	4\alpha_i^6a^2b^2c^2\right)^{-\frac{1}{4}},$$
	we would like approximate $|\vphi(a,b,c)|$ by $e^{-\frac{\norm{\alpha}_2^2}{2}\left(a^2+b^2+c^2\right)}$. For that, we note that
	$$\left||\vphi(a,b,c)| - e^{\frac{-\norm{\alpha}_2^2}{2}\left(a^2+b^2+c^2\right)} \right| \leq \left|\ln\left(|\vphi(a,b,c)|\right) - \ln\left(e^{\frac{-\norm{\alpha}_2^2}{2}\left(a^2+b^2+c^2\right)}\right) \right|.$$
	Since $\left|\ln(x+1)-x\right|\leq x^2$, similar considerations as in $\eqref{char taylor}$, show for $(a,b,c) \in K$:
	\begin{align*}
	\ln(|\varphi|) =& -\frac{1}{4}\sum \ln \left(1 + 2\alpha_i^2\left(a^2+b^2+c^2\right)+\alpha_i^4\left(a^2+b^2+c^2\right)^2+4\alpha_i^6a^2b^2c^2\right)\nonumber\\
	=& -\frac{\norm{\alpha}_2^2}{2}\left(a^2+b^2+c^2\right) - \frac{\norm{\alpha}_4^4}{4}\left(a^2+b^2+c^2\right)^2 - \norm{\alpha}_6^6a^2b^2c^2\nonumber\\
	+& O\left(\norm{\alpha}_4^4\right)\left(\left(a^2+b^2+c^2\right)^2 + \left(a^2+b^2+c^2\right)^4 + a^4b^4c^4\right)\nonumber\\
	=&-\frac{\norm{\alpha}_2^2}{2}\left(a^2+b^2+c^2\right) + O\left(\norm{\alpha}_4^4\right)\left(a^2+b^2+c^2\right)^2.
	\end{align*} 
	This shows the existence of an absolute constant $C_2 >0$ such that for $(a,b,c) \in K$
	$$ \left||\vphi(a,b,c)| - e^{\frac{-\norm{\alpha}_2^2}{2}\left(a^2+b^2+c^2\right)} \right|\leq C_2\norm{\alpha}_4^4\left(a^2+b^2+c^2\right)^2.$$
	Hence
	\begin{align} \label{second approx}
	&\int\limits_{K}|\vphi(a,b,c)|\cos(2\pi t(a+b+c))dadbdc \nonumber\\
	\geq &\int\limits_{K}e^{-\frac{\norm{\alpha}_2^2}{2}\left(a^2+b^2+c^2\right)}\cos\left(2\pi t(a+b+c)\right)dadbdc - C_2\norm{\alpha}^4_4\int\limits_{K}\left(a^2+b^2+c^2\right)^2dadbdc,
	\end{align}
	and 
	$$C_2\norm{\alpha}^4_4\int\limits_{K}\left(a^2+b^2+c^2\right)^2dadbdc \leq C_3\frac{\norm{\alpha}^4_4}{\norm{\alpha}^{77/12}_2}\leq C_3\left(\frac{\norm{\alpha}_3}{\norm{\alpha}_2}\right)^3\frac{1}{\norm{\alpha}_2^3},$$
	for an absolute constant $C_3>0$.
	By rotational invariance of $e^{-\frac{\norm{\alpha}_2^2}{2}\left(a^2+b^2+c^2\right)}$, we may apply $R$ as a unitary coordinate change, which shows
	\begin{align} \label{coor change}
	&\int\limits_{K}e^{-\frac{\norm{\alpha}_2^2}{2}\left(a^2+b^2+c^2\right)}\cos(2\pi t(a+b+c))dadbdc =\int\limits_{R^{-1}K}e^{-\frac{\norm{\alpha}_2^2}{2}\left(a^2+b^2+c^2\right)}\cos(2\sqrt{3}\pi ta)dadbdc \nonumber\\
	=&\int\limits_{-\frac{1}{\norm{\alpha}^{11/12}_2}}^{\frac{1}{\norm{\alpha}^{11/12}_2}}e^{-\frac{\norm{\alpha}_2^2}{2}c^2}dc\int\limits_{-\frac{1}{\norm{\alpha}^{11/12}_2}}^{\frac{1}{\norm{\alpha}^{11/12}_2}}e^{-\frac{\norm{\alpha}_2^2}{2}b^2}db\int\limits_{-\frac{1}{\norm{\alpha}^{11/12}_2}}^{\frac{1}{\norm{\alpha}^{11/12}_2}}e^{-\frac{\norm{\alpha}_2^2}{2}a^2}\cos(\sqrt{12}\pi t a)da\nonumber \\
	=&\frac{1}{\norm{\alpha}^3_2}\int\limits_{-\norm{\alpha}_2^{1/12}}^{\norm{\alpha}_2^{1/12}}e^{-\frac{c^2}{2}}dc \int\limits_{-\norm{\alpha}_2^{1/12}}^{\norm{\alpha}_2^{1/12}}e^{-\frac{b^2}{2}}db \int\limits_{-\norm{\alpha}_2^{1/12}}^{\norm{\alpha}_2^{1/12}}e^{-\frac{a^2}{2}}\cos\left(\sqrt{12}\pi \frac{t}{\norm{\alpha}_2}a\right)da,
	\end{align}
	where the last equality is a result of a second coordinate change. By Lemma \ref{lem: tp}, we know that  $$|t_p| - \frac{k_p}{\norm{\alpha}_2}\leq \left|\frac{t}{\norm{\alpha}_2}\right| \leq |t_p| + \frac{k_p}{\norm{\alpha}_2}$$ for constants $k_p,t_p$ depending on $p$. Also, a calculation shows that 
	$$\int\limits_{-\infty}^{\infty}e^{-\frac{a^2}{2}}\cos\left(\sqrt{12}\pi \frac{t}{\norm{\alpha}_2}a\right)da=\sqrt{2\pi}e^{-6\pi^2\frac{t^2}{\norm{\alpha}_2^2}}.$$
	Note that if $\norm{\alpha}_2^{2/12} \geq 12\pi^2\left(|t_p| + k_p\right)^2 + 2$, then
	$$\int\limits_{|a| > \norm{\alpha}_2^{1/12}}e^{-\frac{a^2}{2}}da \leq 2 e^{-\frac{\norm{\alpha}_2^{2/12}}{2}}\leq \frac{2}{e} e^{-12\pi^2\frac{\left(|t_p| + k_p\right)^2}{2}} \leq \frac{1}{2}\sqrt{2\pi}e^{-6\pi^2\frac{t^2}{\norm{\alpha}_2^2}}.$$ 
	That is, if $\norm{\alpha}_2^{2/12}$ is larger than some constant, which depends only on $t$, we have
	$$\int\limits_{-\norm{\alpha}_2^{1/12}}^{\norm{\alpha}_2^{1/12}}e^{-\frac{a^2}{2}}\cos\left(\sqrt{12}\pi \frac{t}{\norm{\alpha}_2}a\right)da \geq \frac{1}{2}\sqrt{2\pi}e^{-6\pi^2\frac{t^2}{\norm{\alpha}_2^2}}.$$
	Together with the observation $\int\limits_{-1}^{1}e^{\frac{-x^2}{2}}dx > 1$, this shows that the expression \eqref{coor change} is bounded from below by $\frac{1}{2}\sqrt{2\pi}e^{-6\pi^2\frac{t^2}{\norm{\alpha}_2^2}}$.
	Combining the above, along with \eqref{first approx} and \eqref{second approx} shows 
	\begin{align*}
	&\int\limits_{K}\frac{-\mathrm{Im}\left(\varphi(a,b,c)\right)}{abc}\cos(2\pi t(a+b+c))dadbdc \\
	\geq\norm{\alpha}_3^3&\int\limits_{K}\cos(2\pi t(a+b+c))\left|\vphi(a,b,c)\right|dadbdc
	- 2C_1\left(\frac{\norm{\alpha}_3}{\norm{\alpha}_2}\right)^3\frac{1}{\norm{\alpha}_2^{0.5}}\\
	\geq\norm{\alpha}_3^3&\int\limits_{K}e^{-\frac{\norm{\alpha}_2^2}{2}\left(a^2+b^2+c^2\right)}\cos(2\pi t(a+b+c))dadbdc - C_3\left(\frac{\norm{\alpha}_3}{\norm{\alpha}_2}\right)^6- 2C_1\left(\frac{\norm{\alpha}_3}{\norm{\alpha}_2}\right)^3\frac{1}{\norm{\alpha}_2^{0.5}} \\
	\geq\frac{1}{2}&\sqrt{2\pi}e^{-6\pi^2\frac{t^2}{\norm{\alpha}_2^2}}\left(\frac{\norm{\alpha}_3}{\norm{\alpha}_2}\right)^3 - C_3\left(\frac{\norm{\alpha}_3}{\norm{\alpha}_2}\right)^6- 2C_1\left(\frac{\norm{\alpha}_3}{\norm{\alpha}_2}\right)^3\frac{1}{\norm{\alpha}_2^{0.5}} \geq 4\delta'_p\left(\frac{\norm{\alpha}_3}{\norm{\alpha}_2}\right)^3.
	\end{align*}
	whenever $\norm{\alpha}_2^2>c_p''$, for $c_p'',\delta'_p$ constants, depending only on $p$.
	From \eqref{step 2}, we can choose a constant $c'_p>c''_p>0$ such that
	$$\int\limits_{\RR^3 \setminus B_2}\left| \text{Re}\left(\varphi(a,b,c)\widehat{\mathrm{sgn}}(a,b,c)\right)\right|dadbdc < 2\delta'_p\left(\frac{\norm{\alpha}_3}{\norm{\alpha}_2}\right)^3,$$
	whenever $\norm{\alpha}_2^2 >c'_p$. Thus $$I''_p > \int\limits_{K}\frac{-\text{Im}\left(\varphi(a,b,c)\right)}{abc}dadbdc - \int\limits_{\RR^3 \setminus B_2}\left|\frac{\text{Im}\left(\varphi(a,b,c)\right)}{abc}\right|dadbdc \geq 2\delta'_p\left(\frac{\norm{\alpha}_3}{\norm{\alpha}_2}\right)^3.$$\\
\end{proof}
It now remains to show that the difference between $I'_p$ and $(2p-1)^3$ is small, compared to $I''_p$.
\begin{claim} \label{I'}
		Fix $p \in (0,1)$, there exists a constant $c_p>0$ depending only on $p$ such that whenever $\norm{\alpha}_2^2>c_p$ then $\vert I'_p - (2p-1)^3\vert \leq \delta'_p\left(\frac{\norm{\alpha}_3}{\norm{\alpha}_2}\right)^3$, where $\delta'_p$ is the same as in Claim \ref{I''}. 
\end{claim}
\begin{proof}
	Let $g$ be the density of the coordinate free version of $f$, as in Lemma \ref{lem: coordinate free version}, and let $\psi$ be its characteristic function \eqref{define psi}. Evidently, we have the equality:
	\begin{align*}
	&\frac{1}{\pi^3}\mint{\times}\limits_{\RR^3}\psi(a,b,c)\widehat{\mathrm{sgn}}(a,b,c)e^{-2\pi \mathrm{\bf{i}}t(a+b+c)}dadbdc= (2p-1)^3.
	\end{align*}
	Thus, by rewriting $I'_p$ as
	$$\frac{1}{\pi^3}\mint{\times}\limits_{\RR^3} (\mathrm{Re} \left(\varphi(a,b,c)\right)+\psi(a,b,c) -\psi(a,b,c))\frac{\sin(2\pi t(a+b+c))}{abc}dadbdc,$$
	we obtain
	$$I'_p = (2p-1)^3 + \frac{1}{\pi^3}\mint{\times}\limits_{\RR^3}(\mathrm{Re} \left(\varphi(a,b,c)\right) -\psi(a,b,c))\frac{\sin(2\pi t(a+b+c))}{abc}dadbdc.$$
	Next, we rewrite $\sin(2\pi t(a+b+c))$ as:
	\begin{align*}
	&\sin(2\pi t a)\sin(2\pi t b)\sin(2\pi t c) + \cos(2\pi t a)\cos(2\pi t b)\sin(2\pi t c) +\\
	& \cos(2\pi t a)\sin(2\pi t b)\cos(2\pi t c) + \sin(2\pi t a)\cos(2\pi t b)\cos(2\pi t c).
	\end{align*}
	Recall
	\begin{align*}
	\vphi(a,b,c)&= \prod\limits_{i=1}^{d}(1 + \alpha_i^2(a^2+b^2+c^2) +2\alpha_i^3 abc\mathrm{\bf{i}})^{-\frac{1}{2}},\\
	\psi(a,b,c) &= \prod\limits_i\left((1+\alpha_i^2a^2)(1+\alpha_i^2b^2)(1+\alpha_i^2c^2)\right)^{-\frac{1}{2}}.
	\end{align*}
	One may now verify that $\mathrm{Re}(\varphi(a,b,c)- \psi(a,b,c))\frac{1}{abc}$ is an odd function. For a function $h$, we've defined $\Delta_ch(a,b,c) = h(a,b,c) + h(a,b,-c).$ Thus,
	$$\Delta_c\left(\frac{\mathrm{Re} (\varphi(a,b,c))-\psi(a,b,c)}{abc}\sin(2\pi t a)\cos(2\pi t b)\cos(2\pi t c)\right) = 0.$$ 
	Looking at the principal value, we see that: 
	\begin{align*}
	\mint{\times}\limits_{\RR^3} \frac{\mathrm{Re} (\varphi(a,b,c))-\psi(a,b,c)}{abc}\sin(2\pi t a)\cos(2\pi t b)\cos(2\pi t c)dadbdc = 0,
	\end{align*}
	and the same can be said for the other similar terms. We are then left to consider an integrable function:
	$$I'_p -(2p-1)^3 = \int\limits_{\RR^3} \frac{\sin(2\pi ta)\sin(2\pi tb)\sin(2\pi tc)}{abc}
	(\mathrm{Re}(\vphi(a,b,c)-\psi(a,b,c))dadbdc.$$
	By making the substitution $a' = \norm{\alpha}_2a, b' = \norm{\alpha}_2b, c' = \norm{\alpha}_2c$, and denoting $t' = \frac{t}{\norm{\alpha}_2}$ the above equals
	$$\int\limits_{\RR^3} \frac{\sin(2\pi t'a')\sin(2\pi t'b')\sin(2\pi t'c')}{a'b'c'}
	\left(\mathrm{Re}(\vphi_1(a',b',c')-\psi_1(a',b',c'))\right)da'db'dc',$$
	where $\varphi_1$ and $\psi_1$ are as in Lemma \ref{lem: coordinate free version}.
	By Lemma \ref{lem: tp}, we know that $|t'| < |t_p| + \frac{k_p}{\norm{\alpha}_2}$. Thus
	$$ \sup\limits_{(a',b',c')\in \RR^3}\left|\left(\frac{\sin(2\pi t'a')\sin(2\pi t'b')\sin(2\pi t'c)}{a'b'c'}\right)\right|\leq \left(2\pi\left(|t_p| + \frac{k_p}{\norm{\alpha}_2}\right)\right)^3.$$
	And so
	\begin{equation*}
	|I'_p -(2p-1)^3|\leq \left(2\pi\left(|t_p| + \frac{k_p}{\norm{\alpha}_2}\right)\right)^3\int\limits_{\RR^3}
	\left|\mathrm{Re}(\vphi_1(a',b',c'))-\psi_1(a',b',c')\right|da'db'dc'.
	\end{equation*}
	Lemma \ref{lem: coordinate free version} asserts that $\int\limits_{\RR^3}
	|\mathrm{Re}(\vphi_1)-\psi_1|\leq C\left(\frac{\norm{\alpha}_3}{\norm{\alpha}_2}\right)^{3+\eps}$ for large enough $\norm{\alpha}_2^2$.
	Thus,
	$$I'_p - (2p-1)^3 \leq \left(2\pi\left(|t_p| + \frac{k_p}{\norm{\alpha}_2}\right)\right)^3C\left(\frac{\norm{\alpha}_3}{\norm{\alpha}_2}\right)^{3+\eps}.$$
	Since we've assumed $\alpha$ to be normalized as in \eqref{normalization}, $\frac{\norm{\alpha}_3}{\norm{\alpha}_2}$ can be made as small as needed. The proof concludes by choosing $c_p > c'_p$ to be such that
	\begin{equation*}
	\left(2\pi\left(|t_p| + \frac{k_p}{\norm{\alpha}_2}\right)\right)^3C\left(\frac{\norm{\alpha}_3}{\norm{\alpha}_2}\right)^{3+\eps} < \delta'_p\left(\frac{\norm{\alpha}_3}{\norm{\alpha}_2}\right)^{3} \mbox{ whenever } \norm{\alpha}_2^2>c_p.
	\end{equation*}
\end{proof}

Combining Claims \ref{I'} and \ref{I''}, we have thus established
\begin{lemma} \label{lem: lower bound general}
	Fix $p \in (0,1)$. There exist constants $\delta'_p,c_p>0$ depending only on $p$ such that whenever $\norm{\alpha}_2^2>c_p$ then $I_p  \geq (2p-1)^3 + \delta'_p\left(\frac{\norm{\alpha}_3}{\norm{\alpha}_2}\right)^3$. 
\end{lemma}
  Now, by definition $\PP(\langle X_1,X_2 \rangle > t_{p,\alpha})=p$ and $\PP(\langle X_1,X_2 \rangle>t_{p,\alpha},\langle X_1,X_3 \rangle > t_{p,\alpha}) = p^2$. We note that Lemma \ref{lem: lower bound general}, along with \eqref{invere char function} produces:
  $$(2p-1)^3 + \delta'_p\left(\frac{\norm{\alpha}_3}{\norm{\alpha}_2}\right)^3 \leq 8\PP(E_p)-12p^2+6p-1.
  $$
  This proves the lower bound of Theorem \ref{thm: prob estimation}:
  $$p^3 + \frac{\delta'_p}{8}\left(\frac{\norm{\alpha}_3}{\norm{\alpha}_2}\right)^3 \leq \PP(E_p).$$
  \subsection{Upper bound}
  To finish the proof of Theorem \ref{thm: prob estimation} it remains to prove the upper bound. This is done in the following lemma.
  \begin{lemma} \label{lem: upper bound}
  Let $p \in (0,1)$, $\PP(E_p)-p^3 \leq \Delta\left(\frac{\norm{\alpha}_3}{\norm{\alpha}_2}\right)^3$, for a universal constant $\Delta>0$.
  \end{lemma}
  \begin{proof}
  The proof of this lemma will use the higher dimensional analogue of the Berry-Esseen's inequality.\\
 \\
  Define the random vector $V = (\langle X_1, X_2 \rangle, \langle X_1, X_3 \rangle, \langle X_2, X_3 \rangle)$. It is straightforward to check that the covariance matrix of $V$ is $\norm{\alpha}^2_2 \mathrm{I}_3$ where $\mathrm{I}_3$ is the identity matrix. We decompose $V$ into $V_i = \left(X_1^iX_2^i, X_1^iX_3^i,X_2^iX_3^i\right)$. Clearly $V = \sum\limits_{i=1}^dV_i$ and, since $X_1^i,X_2^i,X_3^i$ are i.i.d. Gaussians,
  \begin{align*}
  \EE\norm{V_i}^3 \leq& \sqrt{\EE\left[\left((X_1^iX_2^i)^2+(X_1^iX_3^i)^2+(X_2^iX_3^i)^2\right)^3\right]}\\ =& \sqrt{3\EE[(X_1^iX_2^i)^6]+18\EE[(X_1^i)^6(X_2^i)^4(X_3^i)^2] + 6\EE[(X_1^i)^4(X_2^i)^4(X_3^i)^4]} \leq 50\sqrt{\alpha_i^6} = 50\alpha_i^3.
  \end{align*}
  Thus, if $Z_3$ a $3$-dimensional standard Gaussian random vector, by \eqref{multi-berry-eseen} there is a constant $C_{be}$ such that for any convex set $K \subset \RR^3$ we have that
  $$|\PP (V/\norm{\alpha}_2 \in K) - \PP(Z_3 \in K)| \leq  100C_{be}\left(\frac{\norm{\alpha}_3}{\norm{\alpha}_2}\right)^3.$$
  In particular, this holds for the convex set 
  $$E_p = \left\{(x,y,z) \in \RR^3|x >\frac{t_{p,\alpha}}{\norm{\alpha}_2},\ y >\frac{t_{p,\alpha}}{\norm{\alpha}_2},\ z >\frac{t_{p,\alpha}}{\norm{\alpha}_2}\right\}.$$ 
  If we denote $p' = \Phi^{-1}(\frac{{t_{p,\alpha}}}{\norm{\alpha}_2})$ , the above shows 
  $$|\PP(V/\norm{\alpha}_2 \in E_p) - p'^3| \leq  100C_{be}\left(\frac{\norm{\alpha}_3}{\norm{\alpha}_2}\right)^3.$$
 By Lemma \ref{lem: tp}, $|p - p'| \leq 3\left(\frac{\norm{\alpha}_3}{\norm{\alpha}_2}\right)^3$. Also
  $$|p^3-p'^3| = |p-p'|(p^2+pp'+p'^2)\leq 9\left(\frac{\norm{\alpha}_3}{\norm{\alpha}_2}\right)^3.$$
  We then have $$|\PP(E_p)-p^3| \leq |\PP(E_p)-p'^3| + |p^3 - p'^3| \leq (9+100C_{be})\left(\frac{\norm{\alpha}_3}{\norm{\alpha}_2}\right)^3$$ as desired.
  \end{proof}
 \section{Proof of Theorem \ref{thm: signed triangles}}
 Recall from the introduction that $\tau(G)$ denotes the number of signed triangles of a graph $G$. If $A$ is the adjacency matrix of $G$ with entries $A_{i,j}$ we denote the centered adjacency matrix of $G$ as $\bar{A}$ with entries $\bar{A}_{i,j} := A_{i,j} - \EE[A_{i,j}]$. Given three distinct vertices $i$,$j$ and $k$, the signed triangle induced by those 3 vertices is $\tau_G(i,j,k) := \bar{A}_{i,j}\bar{A}_{i,k}\bar{A}_{j,k}$. It then holds that for a graph $G=(V,E)$ the number of signed triangles is given by:
 $$\tau(G) := \sum\limits_{\{i,j,k\}\in\binom{V}{3}}\tau_G(i,j,k).$$
Analysis of $\tau(G(n,p))$ was done in ~\cite{BDER14}, where it was shown that $\EE \tau(G(n,p)) = 0$ while $\mathrm{Var}(\tau(G(n,p))) \leq n^3$.\\
\\
To prove Theorem \ref{thm: signed triangles} it will suffice to show that $\EE \tau(G(n,p,\alpha))$ is asymptotically bigger than both the standard deviation of $\tau(G(n,p))$ and of $\tau(G(n,p,\alpha))$, provided that $\left(\frac{\norm{\alpha}_2}{\norm{\alpha}_3}\right)^6 << n^3$.\\
\\
For this aim we first prove some technical lemmas:
\begin{lemma} \label {lem: wedge corrleation bound}
	Let $p \in (0,1)$, then 
	$$\EE A_{1,2}A_{2,3} \leq p^2 + 8\left(\frac{\norm{\alpha}_4}{\norm{\alpha}_2}\right)^4.$$
\end{lemma}
\begin{proof}
	Let $X,Y,Z$ be i.i.d. random variables generated from $\mathcal{N}(0,D_\alpha)$, then, conditioning on $Y$ yields the expression
	$$\EE A_{1,2}A_{2,3} = \EE\left[\PP\left(\langle X, Y\rangle \geq t_{p,\alpha}\right)\PP\left(\langle Z, Y\rangle \geq t_{p,\alpha}\right)|Y\right] = \EE\left[\Phi\left(\frac{t_{p,\alpha}}{\sqrt{\sum \alpha_iY_i^2}}\right)^2\right],$$
	where $\Phi$ is the standard Gaussian cumulative distribution function. By the same argument, we also have
	$$\EE\left[\Phi\left(\frac{t_{p,\alpha}}{\sqrt{\sum \alpha_iY_i^2}}\right)\right]^2 = \EE A_{1,2}\EE A_{2,3} = p^2.$$
	Thus, it will be enough to show,
	\begin{equation} \label{eq: CDF bound}
	\mathrm{Var}\left(\Phi\left(\frac{t_{p,\alpha}}{\sqrt{\sum \alpha_iY_i^2}}\right)\right) \leq 8\left(\frac{\norm{\alpha}_4}{\norm{\alpha}_2}\right)^4.
	\end{equation}
	For $\sigma^2 > 0$ denote by $G_{\sigma^2}$ a random variable with law $\mathcal{N}(0,\sigma^2)$. We then have
	$$\Phi\left(\frac{t_{p,\alpha}}{\sqrt{\sum \alpha_iY_i^2}}\right) = \PP\left(G_{\sum\alpha_iY_i^2} \leq t_{p,\alpha}\right),$$
	and
\begin{align*}
	\mathrm{Var}\left(\Phi\left(\frac{t_{p,\alpha}}{\sqrt{\sum \alpha_iY_i^2}}\right)\right) &\leq \EE\left[\left(\Phi\left(\frac{t_{p,\alpha}}{\sqrt{\sum \alpha_iY_i^2}}\right) - \Phi\left(\frac{t_{p,\alpha}}{\norm{\alpha}_2}\right)\right)^2\right] \\
	&=\EE\left[\left(\PP\left(G_{\sum\alpha_iY_i^2} \leq t_{p,\alpha}\right) -\PP\left(G_{\norm{\alpha}_2^2} \leq t_{p,\alpha}\right)\right)^2\right] \\
	&\leq \EE\left[\mathrm{TV}\left(G_{\sum\alpha_iY_i^2},G_{\norm{\alpha}_2^2}\right)^2\right].
\end{align*}
For the total variation distance between $2$ Gaussian random variables we have the following bound (see Proposition 3.6.1 in \cite{nourdin2012normal}, for example):
$$\mathrm{TV}\left(G_{\sigma^2_1},G_{\sigma^2_1}\right) \leq 2\frac{|\sigma^2_1-\sigma^2_2|}{\max(\sigma_1^2,\sigma_2^2)}.$$
This implies
$$\mathrm{Var}\left(\Phi\left(\frac{t_{p,\alpha}}{\sqrt{\sum \alpha_iY_i^2}}\right)\right) \leq \frac{4}{\norm{\alpha}_2^4}\EE\left[\left(\sum\alpha_iY_i^2 - \norm{\alpha}_2^2\right)^2\right].$$
As $Y\sim \mathcal{N}(0,D_\alpha)$, it is immediate to check
$$\EE\left[\sum\alpha_iY_i^2\right] = \sum\alpha_i^2 = \norm{\alpha}_2^2.$$
Hence
$$\EE\left[\left(\sum\alpha_iY_i^2 - \norm{\alpha}_2^2\right)^2\right] = \sum\alpha_i^2\mathrm{Var}\left(Y_i^2\right) = 2\norm{\alpha}_4^4.$$
This establishes \eqref{eq: CDF bound} and finishes the proof.

\end{proof}
\begin{lemma} \label {lem: triangle corrleation bound}
Let $p \in (0,1)$, then 
$$\EE[\tau_{G(n,p,\alpha)}(1,2,3)\tau_{G(n,p,\alpha)}(1,2,4)] \leq 80\left(\frac{\norm{\alpha}_4}{\norm{\alpha}_2}\right)^4$$
\end{lemma}
\begin{proof}
	The proof is similar to Lemma \ref{lem: wedge corrleation bound} and uses the observation that if $V_1$ and $V_2$ are the random vectors corresponding to two vertices, then conditioned on their values, the random variables $\tau_{G(n,p,\alpha)}(1,2,3)$ and $\tau_{G(n,p,\alpha)}(1,2,4)$ are independent. Thus
	\begin{align*} 
	\EE[\tau_{G(n,p,\alpha)}(1,2,3)\tau_{G(n,p,\alpha)}(1,2,4)] &= \EE\left[\EE[\tau_{G(n,p,\alpha)}(1,2,3)\tau_{G(n,p,\alpha)}(1,2,4)|V_1,V_2]\right] \\
	&=\EE\left[\EE\left[\bar{A}_{1,3}\bar{A}_{2,3}|V_1,V_2\right]^2\bar{A}^2_{1,2}\right] \leq \EE\left[\EE\left[\bar{A}_{1,3}\bar{A}_{2,3}|V_1,V_2\right]^2\right]. 
	\end{align*}
	Lemma \ref{lem: wedge corrleation bound} implies $\EE\left[\bar{A}_{1,3}\bar{A}_{2,3}\right] \leq 8\left(\frac{\norm{\alpha}_4}{\norm{\alpha}_2}\right)^4$,  so that
	\begin{align}\label{eq: triangle corr core}
	\EE[\tau_{G(n,p,\alpha)}(1,2,3)\tau_{G(n,p,\alpha)}(1,2,4)] &\leq\EE\left[\EE\left[\bar{A}_{1,3}\bar{A}_{2,3}|V_1,V_2\right]^2\right]\nonumber \\
	&\leq \mathrm{Var}\left(\EE\left[\bar{A}_{1,3}\bar{A}_{2,3}|V_1,V_2\right]\right) + 64\left(\frac{\norm{\alpha}_4}{\norm{\alpha}_2}\right)^{8}.
	\end{align}
	Note that if $X \sim \mathcal{N}(0,D_\alpha)$, 
	\begin{align*} 
	&\EE\left[\bar{A}_{1,3}\bar{A}_{2,3}|V_1,V_2\right] = \\
	&(1-p)^2\mathbb{P}\left(\langle X,V_1 \rangle \geq t_{p,\alpha},\langle X,V_2 \rangle \geq t_{p,\alpha}\right) + p^2\mathbb{P}\left(\langle X,V_1 \rangle < t_{p,\alpha},\langle X,V_2 \rangle < t_{p,\alpha}\right)\\
	-&p(1-p)\left(\mathbb{P}\left(\langle X,V_1 \rangle \geq t_{p,\alpha},\langle X,V_2 \rangle < t_{p,\alpha}\right) + \mathbb{P}\left(\langle X,V_1 \rangle < t_{p,\alpha},\langle X,V_2 \rangle \geq t_{p,\alpha}\right)\right).
	\end{align*}
	For any $v,u \in \RR^d$, denote by $\Sigma_{v,u}$ the matrix given by
	$$\left[ {\begin{array}{cc}
		\sum\alpha_iv_i^2 & \sum\alpha_iv_iu_i \\
		\sum\alpha_iv_iu_i & \sum\alpha_iu_i^2 \\
		\end{array} } \right], $$
	 then the joint law of $ \langle X,v \rangle,\langle X,u \rangle$ is $G_{v,u} := \mathcal{N}\left(0,\Sigma_{v,u}\right).$
	 The above can now be rewritten as
	\begin{align} \label{eq: subsets representation}
	&(1-p)^2\mathbb{P}\left(G_{V_1,V_2} \in (t_{p,\alpha}, \infty) \times (t_{p,\alpha}, \infty)\right) + p^2\mathbb{P}\left(G_{V_1,V_2} \in (-\infty, t_{p,\alpha}) \times (-\infty, t_{p,\alpha})\right) \nonumber\\
	-&p(1-p)\left(\left(G_{V_1,V_2} \in (-\infty,t_{p,\alpha}) \times (t_{p,\alpha}, \infty)\right) + \mathbb{P}\left(G_{V_1,V_2} \in (t_{p,\alpha}, \infty) \times (-\infty, t_{p,\alpha})\right)\right).
	\end{align}
	In particular, if $M_i$ are independent random Wishart matrices with law $\mathcal{W}_2(\alpha^2_i\mathrm{I}_2,1)$ and $M = \sum M_i$ then the matrix $\Sigma_{V_1,V_2}$ has the same law as $M$. In this case we regard \eqref{eq: subsets representation} as a function $h$, of the covariance $M$.  
	Using \eqref{eq: triangle corr core}, we get
	$$\EE[\tau_{G(n,p,\alpha)}(1,2,3)\tau_{G(n,p,\alpha)}(1,2,4)] \leq \mathrm{Var}\left(h(M)\right) + 16\left(\frac{\norm{\alpha}_4}{\norm{\alpha}_2}\right)^{8}.$$
	It is thus enough to establish an upper bound for $\mathrm{Var}(h(M))$. 
	For a positive semi-definite matrix $\Sigma$, we denote  $G_\Sigma \sim \mathcal{N}(0,\Sigma)$. As $h(M) \leq 1$, we have the following inequality
	\begin{align*}
	\mathrm{Var}(h(M)) &\leq \EE\left[\left(h(M) - h(\norm{\alpha}_2^2\mathrm{I}_2)^2\right)\right]\leq \EE\left[\mathrm{TV}\left(G_M, G_{\norm{\alpha}_2^2\mathrm{I}}\right)^2\right] \\
	&\leq \EE\left[\min\left(1, \mathrm{Ent}\left(G_M||G_{\norm{\alpha}_2^2\mathrm{I}_2}\right)\right)\right],
	\end{align*}
	where we have used Pinsker's inequality \eqref{pinsker} to bound the total variation.
	The relative entropy between the Gaussians (see \cite{duchi2007derivations}) is given by
	\begin{align*}
	\mathrm{Ent}\left(G_M||G_{\norm{\alpha}_2^2\mathrm{I}_2}\right) &=  \mathrm{Tr}\left({\frac{1}{\norm{\alpha}_2^2}M}\right) - \ln\left(\frac{\det(M)}{\det(G_{\norm{\alpha}_2^2\mathrm{I}_2})}\right) - 2 \\
	&= \mathrm{Tr}\left(\frac{M}{\norm{\alpha}_2^2}-\mathrm{I}_2\right) - \ln\left(\det\left(\frac{M}{\norm{\alpha}_2^2}\right)\right).
	\end{align*}
	For any $x \geq \frac{1}{2}$ we have the inequality
	$ x-1 - \ln(x) \leq x^2$. So, if both eigenvalues of $\frac{M}{\norm{\alpha}_2^2}$ are bigger than $\frac{1}{2}$
	$$\mathrm{Ent}\left(G_M||G_{\norm{\alpha_2^2}\mathrm{I}_2}\right) \leq \norm{\frac{M}{\norm{\alpha}_2^2} - \mathrm{I}_2}_{HS}^2.$$
	Otherwise, it is clear that
	$$1 \leq 2\norm{\frac{M}{\norm{\alpha}_2^2} - \mathrm{I}_2}_{HS}^2.$$
	Combining the above, we have established
	\begin{equation} \label{eq: var to HS}
	\mathrm{Var}\left(h(M)\right) \leq 2\EE\left[\norm{\frac{M}{\norm{\alpha}_2^2} - \mathrm{I}_2}_{HS}^2\right].
	\end{equation}
	Recall that the diagonal elements of $M$ are given by $\sum\alpha_i(V)_i^2$ where $V \sim \mathcal{N}(0,D_\alpha)$, thus
	$$\EE \left[M_{1,1}\right] =  \EE\left[\left(\sum\alpha_i(V_1)_i^2\right)\right] = \left(\sum \alpha_i^2\right)= \norm{\alpha}_2^2,$$
	and
	$$\mathrm{Var}(M_{1,1}) = \sum\mathrm{Var}\left(\alpha_i(V_1)_i^2\right)  = 2\norm{\alpha}_4^4.$$
	The off-diagonal element is given by $\sum\alpha_i\left(V_1\right)_i\left(V_2\right)_i$, for which we have
	$$\EE\left[M_{1,2}\right] = 0,$$
	and
	$$\mathrm{Var}\left(M_{1,2}\right) = \EE\left[\left(\sum\alpha_i\left(V_1\right)_i\left(V_2\right)_i\right)^2\right] = \sum\alpha_i^2\EE\left[\left(V_1\right)^2_i\right]\EE\left[\left(V_2\right)^2_i\right] = \norm{\alpha}_4^4.$$
	Using these estimates in \eqref{eq: var to HS} we obtain
	$$\mathrm{Var}\left(h(M)\right) \leq 16\left(\frac{\norm{\alpha}_4}{\norm{\alpha}_2}\right)^4.$$
	plugging this into \eqref{eq: triangle corr core} gives the desired result.
\end{proof}
\begin{lemma} \label{lem: star correlation bound}
	Let $p \in (0,1)$, then 
	$$\EE[\tau_{G(n,p,\alpha)}(1,2,3)\tau_{G(n,p,\alpha)}(1,4,5)] \leq 80\left(\frac{\norm{\alpha}_4}{\norm{\alpha}_2}\right)^4$$
\end{lemma}
\begin{proof}
	Conditioned on the location of the vertex $V_1$, the random variables $\tau_{G(n,p,\alpha)}(1,2,3),\tau_{G(n,p,\alpha)}(1,4,5)$ are independent and identically distributed, thus
	\begin{align*}
	\EE[\tau_{G(n,p,\alpha)}(1,2,3)\tau_{G(n,p,\alpha)}(1,4,5)] &= \EE\left[\EE\left[\tau_{G(n,p,\alpha)}(1,2,3)\tau_{G(n,p,\alpha)}(1,4,5)|V_1\right]\right]\\
	&= \EE\left[\EE\left[\tau_{G(n,p,\alpha)}(1,2,3)|V_1\right]^2\right].
	\end{align*}
	Using the tower property of conditional expectation
	\begin{align*}
	\EE\left[\EE\left[\tau_{G(n,p,\alpha)}(1,2,3)|V_1\right]^2\right] &=\EE\left[\EE\left[\EE\left[\tau_{G(n,p,\alpha)}(1,2,3)|V_1,V_2\right]|V_1\right]^2\right] \\
	&\leq \EE\left[\EE\left[\tau_{G(n,p,\alpha)}(1,2,3)|V_1,V_2\right]^2\right] = \EE\left[\EE\left[\tau_{G(n,p,\alpha)}(1,2,3)|V_1,V_2\right]^2\right].
	\end{align*}
	But in Lemma \ref{lem: triangle corrleation bound}, using \eqref{eq: triangle corr core}, we have essentially shown
	$$\EE\left[\EE\left[\tau_{G(n,p,\alpha)}(1,2,3)|V_1,V_2\right]^2\right] \leq 80\left(\frac{\norm{\alpha}_4}{\norm{\alpha}_2}\right)^4,$$
	thus the claim is proven.
\end{proof}
Towards the proof of Theorem \ref{thm: main} we now estimate $\EE\tau(G(n,p,\alpha))$. Note that since 
$$\EE\tau(G(n,p,\alpha)) = \binom{n}{3}\EE\tau_{G(n,p,\alpha)}(1,2,3),$$
 it is enough to estimate $\EE\tau_{G(n,p,\alpha)}(1,2,3)$.
\begin{align} \label{expectation}
\nonumber \EE\tau_{G(n,p,\alpha)}(1,2,3) &= \EE\bar{A}_{1,2}\bar{A}_{1,3}\bar{A}_{2,3} = \EE (A_{1,2}-p)(A_{1,3}-p)(A_{2,3}-p)\\ \nonumber
&= \EE A_{1,2}A_{1,3}A_{2,3} - p\left(\EE A_{1,2}A_{2,3}+\EE A_{1,2}A_{1,3}+\EE A_{1,3}A_{2,3}\right)\\\nonumber
&\ \ \ \ \ \ \ \ +p^2\left(\EE A_{1,2}+\EE A_{1,3}+\EE A_{2,3}\right) - p^3\\
&\geq\EE A_{1,2}A_{1,3}A_{2,3}-p^3  - 24p\left(\frac{\norm{\alpha}_4}{\norm{\alpha}_2}\right)^4,
\end{align}
where the inequality follows from the fact that $\EE A_{i,j} = p$ and Lemma \ref{lem: wedge corrleation bound}. As
$$\left(\frac{\norm{\alpha}_4}{\norm{\alpha}_2}\right)^4 \leq \left(\frac{\norm{\alpha}_3}{\norm{\alpha}_2}\right)^3 \frac{1}{\norm{\alpha}_2}, $$
as long as $\norm{\alpha}_2$ is large enough, the lower bound of Theorem \ref{thm: prob estimation} yields
$$\EE\tau_{G(n,p,\alpha)}(1,2,3) \geq \delta_p \left(\frac{\norm{\alpha}_3}{\norm{\alpha}_2}\right)^3$$
for a constant $\delta_p >0$, depending only on $p$. This shows
$$\EE\tau(G(n,p,\alpha)) \geq \delta_p\binom{n}{3}\left(\frac{\norm{\alpha}_3}{\norm{\alpha}_2}\right)^3.$$
To bound from above the variance of $\tau(G(n,p,\alpha))$ we observe that $\tau_G(i,j,k)$ is independent from $\tau_G(i',j',k')$ whenever $\left|\left\{i,j,k\right\} \cap \left\{i',j',k'\right\}\right| = 0$, thus
\begin{align*}
&\mathrm{Var}\left(\tau(G(n,p,\alpha))\right)\\
&=\sum\limits_{\left\{i,j,k\right\}}\sum\limits_{\left\{i',j',k'\right\}}\EE\left[\tau_{G(n,p,\alpha)}(i,j,k)\tau_{G(n,p,\alpha)}(i',j',k')\right] - \EE\left[\tau_{G(n,p,\alpha)}(i,j,k)\right]\EE\left[\tau_{G(n,p,\alpha)}(i',j',k')\right]\\ 
&\leq \sum\limits_{\{i,j,k\}}\EE\left[\tau_{G(n,p,\alpha)}(i,j,k)\tau_{G(n,p,\alpha)}(i,j,k)\right] + \sum\limits_{\{i,j,k, l\}}\EE\left[\tau_{G(n,p,\alpha)}(i,j,k)\tau_{G(n,p,\alpha)}(i,j,l)\right] \\
&\ \ \ \ \ \ \ \ \ \ \ \ \ \ \ \ \ \ \ \ \ \ \ \ \ \ \ \ \ \ \ \ \ \ \ \ \ \ \ \ \ \ \ \ \ \ \ \ \ \ \ \ \ \ \ \ \ \ \ \ \ \ \ \  \ \ \ \ \ \  + \sum\limits_{\{i,j,k, l,m\}}\EE\left[\tau_{G(n,p,\alpha)}(i,j,k)\tau_{G(n,p,\alpha)}(k,l,m)\right]\\
&= \binom{n}{3}\EE[\tau_{G(n,p,\alpha)}(1,2,3)\tau_{G(n,p,\alpha)}(1,2,3)] + \binom{n}{4}\binom{4}{2}\EE[\tau_{G(n,p,\alpha)}(1,2,3)\tau_{G(n,p,\alpha)}(1,2,4)]\\
&\ \ \ \ \ \ \ \ \ \ \ \ \ \ \ \ \ \ \ \ \ \ \ \ \ \ \ \ \ \ \ \ \ \ \ \ \ \ \ \ \ \ \ \ \ \ \ \ \ \ \ \ \ \ \ \ \ \ \ \ \ \ \ \  \ \ \ \ \   +5\binom{n}{5}\EE[\tau_{G(n,p,\alpha)}(1,2,3)\tau_{G(n,p,\alpha)}(1,4,5)].
\end{align*}
Noting that $\EE[\tau_{G(n,p,\alpha)}(1,2,3)\tau_{G(n,p,\alpha)}(1,2,3)] \leq 1$, in conjunction with Lemmas \ref{lem: triangle corrleation bound} and \ref{lem: star correlation bound} yields
$$\mathrm{Var}(\tau(G(n,p,\alpha)))  \leq n^3 + 80n^5\left(\frac{\norm{\alpha}_4}{\norm{\alpha}_2}\right)^4.$$
Combining all of the above
$$\EE\left[\tau(G(n,p))\right]=0,\ \ \EE[\tau(G(n,p,\alpha))] \geq \delta_p\binom{n}{3}\left(\frac{\norm{\alpha}_3}{\norm{\alpha}_2}\right)^3,$$
and
$$\max\{\mathrm{Var}(\tau(G(n,p,\alpha))),\mathrm{Var}(G(n,p))\} \leq n^3 + 80n^5\left(\frac{\norm{\alpha}_4}{\norm{\alpha}_2}\right)^4.$$
Chebyshev's inequality implies that
$$\PP\left(\tau\left(G(n,p,\alpha)\right) \leq \frac{1}{2}\EE[\tau(G(n,p,\alpha))]\right) \leq 200 \frac{\left(\frac{\norm{\alpha}_2}{\norm{\alpha}_3}\right)^6n^3+80n^5\frac{\norm{\alpha}_2^2\norm{\alpha}_4^4}{\norm{\alpha}_3^6}}{\delta_p^2n^6},$$
and also
$$\PP\left(\tau(G(n,p)) \geq \frac{1}{2}\EE[\tau(G(n,p,\alpha))]\right) \leq 200 \frac{\left(\frac{\norm{\alpha}_2}{\norm{\alpha}_3}\right)^6n^3+80n^5\frac{\norm{\alpha}_2^2\norm{\alpha}_4^4}{\norm{\alpha}_3^6}}{\delta_p^2n^6}.$$
Note that due to the normalization \eqref{normalization} $\frac{\norm{\alpha}_2^2\norm{\alpha}_4^4}{\norm{\alpha}_3^6} \leq \frac{\norm{\alpha}_2^2}{\norm{\alpha}_3^2}$. Putting the above expressions together we thus have:
$$\mathrm{TV}\left(\tau(G(n,p,\alpha)),\tau(G(n,p))\right) \geq 1 - C\frac{\left(\frac{\norm{\alpha}_2}{\norm{\alpha}_3}\right)^6}{n^3} - C\frac{\left(\frac{\norm{\alpha}_2}{\norm{\alpha}_3}\right)^2}{n},$$
for a constant $C$ depending only on $p$. This concludes the proof of Theorem \ref{thm: signed triangles}.

\section{Proof of the lower bound}
As stated in the introduction, we can view $G(n,p,\alpha)$ as a function of an appropriate random matrix, as follows. Let $\mathbb{Y}$ be a random $n\times d$ matrix with rows sampled i.i.d. from $\mathcal{N}(0,D_\alpha)$. Define $W = W(n,\alpha)= \mathbb{Y}\mathbb{Y}^T/\norm{\alpha}_2 -\mathrm{diag}\left(\mathbb{Y}\mathbb{Y}^T/\norm{\alpha}_2\right)$. Note that for $i \neq j$, $W_{ij} = \langle \gamma_i, \gamma_j\rangle/\norm{\alpha}_2$, where $\gamma_i,\gamma_j$ are the rows of $\mathbb{Y}$. Thus the $n\times n$ matrix $A$ defined as
\begin{equation*}
A_{i,j} = \begin{cases}
1 &\text{if $W_{ij}\geq t_{p,\alpha}/\norm{\alpha}_2$ and $i \neq j$}\\
0 &\text{otherwise}
\end{cases}
\end{equation*}
has the same law as the adjacency matrix of $G(n,p,\alpha)$. Denote the map that takes $W$ to $A$ by $H_{p,\alpha}$, i.e., $A = H_{p.\alpha}(W)$.\\
\\
Similarly, we may view $G(n,p)$ as function of an $n\times n$ matrix with independent Gaussian entries. Let $M(n)$ be a symmetric $n\times n$ random matrix with $0$ entries in the diagonal, and whose entries above the diagonal are i.i.d. standard normal random variables. If $\Phi$ is the cumulative distribution function of the standard Gaussian, then the $n\times n$ matrix $B$, defined as
\begin{equation*}
B_{i,j} = \begin{cases}
1 &\text{if $M(n)_{ij}\geq \Phi^{-1}(p)$ and $i \neq j$}\\
0 &\text{otherwise}
\end{cases}
\end{equation*}
has the same law as the adjacency matrix of $G(n,p)$. Denote the map that takes $M(n)$ to $B$ by $K_p$, i.e., $B = K_p(M(n))$.\\
\\
Using the triangle inequality and by the previous two paragraphs, we have that for any $p \in (0,1)$
\begin{align*}
\mathrm{TV}(G(n,p), G&(n,p,\alpha)) = \mathrm{TV}(K_p(M(n)), H_{p,\alpha}(W(n,\alpha)))\\
& \leq \mathrm{TV}(H_{p,\alpha}(M(n)),H_{p,\alpha}(W(n,\alpha))) + \mathrm{TV}(K_p(M(n)), H_{p,\alpha}(M(n)))\\
& \leq \mathrm{TV}(M(n),W(n,\alpha)) + \mathrm{TV}(K_p(M(n)), H_{p,\alpha}(M(n))).
\end{align*}
The second term is of lower order and will be dealt with later. The first term is bounded using Pinsker's inequality , \eqref{pinsker}, yielding
$$ \mathrm{TV}( M(n), W(n,\alpha)) \leq \sqrt{\frac{1}{2}\mathrm{Ent}[M(n)\big|\big|W(n,\alpha)]}.$$
We'll use a similar argument to the one presented in ~\cite{BG15} which follows an inductive proof using the chain rule for relative entropy.
We observe that a sample of $W(n+1,\alpha)$ may be constructed from $W(n,\alpha)$ by adjoining the column vector (and symmetrically the row vector) $\mathbb{Y}Y/\norm{\alpha}_2$ where $Y \sim \mathcal{N}(0,D_\alpha)$ is independent of $\mathbb{Y}$. Thus, using the notation, $Z_n$ for a standard Gaussian in $\RR^n$, by \eqref{chain rule}, we obtain
$$\mathrm{Ent}\left[W(n+1,\alpha)\big|\big|M(n +1)\right] = \mathrm{Ent}\left[W(n,\alpha)\big|\big|M(n)\right] + \EE_{\mathbb{Y}} \mathrm{Ent}\left[\mathbb{Y}Y/\norm{\alpha}_2\big|W(n,\alpha)\big|\big|Z_n\right].$$
Since $W(n,\alpha)$ is a function of $\mathbb{Y}$, standard properties of relative entropy (see ~\cite{cover2012elements}, chapter 2) show
\begin{align*}
&\EE_{\mathbb{Y}} \mathrm{Ent}\left[\mathbb{Y}Y/\norm{\alpha}_2\big|W(n,\alpha)\big|\big|Z_n\right]\\
=&\EE_{\mathbb{Y}} \mathrm{Ent}\left[\mathbb{Y}Y/\norm{\alpha}_2\big|\mathbb{Y}\mathbb{Y}^T/\norm{\alpha}_2\big|\big|Z_n\right] \leq \EE_{\mathbb{Y}} \mathrm{Ent}\left[\mathbb{Y}Y/\norm{\alpha}_2\big|\mathbb{Y}\big|\big|Z_n\right].
\end{align*}
Note that $\mathbb{Y}Y/\norm{\alpha}_2|\mathbb{Y}$ is distributed as $\mathcal{N}(0,\frac{1}{\norm{\alpha}_2^2}\mathbb{Y}D_\alpha\mathbb{Y}^T)$. The relative entropy between two $n$-dimensional Gaussians, (see ~\cite{duchi2007derivations}) $\mathcal{N}_1 \sim \mathcal{N}(0,\Sigma_1),\mathcal{N}_2 \sim \mathcal{N}(0,\Sigma_2)$ is given by
$$\mathrm{Ent}\left[\mathcal{N}_1 || \mathcal{N}_2\right] = \frac{1}{2}\left(\mathrm{tr}\left(\Sigma_2^{-1}\Sigma_1\right) +\ln\left(\frac{\det\Sigma_2}{\det\Sigma_1}\right)-n\right).$$
In our case $\Sigma_2 = \mathrm{I}_n$ and $\EE_{\mathbb{Y}}\  \mathrm{tr}(\mathbb{Y}D_\alpha\mathbb{Y}^T) = n\norm{\alpha}_2^2$. Thus the following holds:
$$\EE_\mathbb{Y}\ \mathrm{Ent}\left[\frac{1}{\norm{\alpha}_2}\mathbb{Y}Y\big|\mathbb{Y}\big|\big|Z_n\right] = -\frac{1}{2}\left(\EE_\mathbb{Y}\ \ln\det\left(\frac{1}{\norm{\alpha}_2^2}\mathbb{Y}D_\alpha\mathbb{Y}^T\right)\right).$$
Theorem \ref{thm: entropy} is then implied by the following lemma:
\begin{lemma} \label{lem: ln bound}
$-\EE_\mathbb{Y}\ \ln\det\left(\frac{1}{\norm{\alpha}_2^2}\mathbb{Y}D_\alpha\mathbb{Y}^T\right) \leq C\left(n^2 \left(\frac{\norm{\alpha}_4}{\norm{\alpha}_2}\right)^4 + \sqrt{n\left(\frac{\norm{\alpha}_4}{\norm{\alpha}_2}\right)^4 }\right)$ for a universal constant $C>0$.
\end{lemma}
The proof will follow similar lines as Lemma 2 in ~\cite{BG15}. Namely, we will decompose the expectation on the event that the smallest eigenvalue of $\frac{1}{\norm{\alpha}_2^2}\mathbb{Y}D_\alpha\mathbb{Y}^T$, denoted by $\lambda_{\mathrm{min}}$, is larger than $\frac{1}{2}$. Lemma \ref{lem: ln bound} will then follow by the following two claims:

\begin{claim}
	$$-\EE_\mathbb{Y}\left[ \ln\det\left(\frac{1}{\norm{\alpha}_2^2}\mathbb{Y}D_\alpha\mathbb{Y}^T\right)\mathbbm{1}_{\left\{\lambda_{\mathrm{min}}\geq \frac{1}{2}\right\}}\right] \leq C\left(n^2 \left(\frac{\norm{\alpha}_4}{\norm{\alpha}_2}\right)^4 + \sqrt{n\left(\frac{\norm{\alpha}_4}{\norm{\alpha}_2}\right)^4 }\right),$$
	for a universal constant $C > 0$.
\end{claim}
\begin{proof}
	We first use the inequality $-\ln(x) \leq 1-x + (1-x)^2$ for $x \geq \frac{1}{2}$:
	\begin{align} \label{HS sum}
	-\EE_\mathbb{Y}&\left[ \ln\det\left(\frac{\mathbb{Y}D_\alpha\mathbb{Y}^T}{\norm{\alpha}_2^2}\right)\mathbbm{1}_{\left\{\lambda_{\mathrm{min}}\geq \frac{1}{2}\right\}}\right] \nonumber\\
	&\leq \EE_\mathbb{Y}\left[\Bigg\vert\mathrm{tr}\left(\mathrm{I}_n-\frac{\mathbb{Y}D_\alpha\mathbb{Y}^T}{\norm{\alpha}_2^2}\right)\Bigg\vert + \norm{\mathrm{I}_n-\frac{\mathbb{Y}D_\alpha\mathbb{Y}^T}{\norm{\alpha}_2^2}}^2_{HS}\right],
	\end{align}
	where $\norm{\cdot}_{HS}$ denotes the Hilbert-Schmidt norm. Before proceeding, we first calculate several quantities. For $1 \leq j \leq n$ denote by $A_j$ the $j^{th}$ row of $\mathbb{Y}\sqrt{D_\alpha}$ with entries $\{\sqrt{\alpha_i}y_{j,i}\}_{i=1}^d$.
	\begin{enumerate}
		\item The expected squared norm of $A_j$ is given by $\EE \norm{A_j}^2 = \sum\limits_i \EE\ \alpha_iy_{j,i}^2 = \sum\limits_i \alpha_i^2 = \norm{\alpha}_2^2$. Since $y_{j,i}$ is a centred Gaussian with variance $\alpha_i$.
		\item When $j \neq k$, $A_j$ and $A_k$ are independent, and so $\EE\ \norm{A_j}^2\norm{A_k}^2 = \left(\sum\limits_i\alpha_i^2\right)^2 = \norm{\alpha}_2^4$.
		\item When $j \neq k$, the expected squared inner product between two rows is given by 
		\begin{align*}
		\EE \langle A_j,A_k\rangle^2 &= \EE\left(\sum\limits_{i = 1}^{d}\alpha_iy_{j,i}y_{k,i}\right)^2 \\&= \sum\limits_{i=1}^{d}\alpha_i^2\EE y_{j,i}^2y_{k,i}^2 + \sum\limits_{i_1 \neq i_2} \alpha_{i_1}\alpha_{i_2}\EE y_{j,i_1}y_{k,i_1}y_{j,i_2}y_{k,i_2} =\sum\limits_{i=1}^d\alpha_i^4 = \norm{\alpha}_4^4.
		\end{align*}
		\item The expected 4th power of the norm is given by  
		\begin{align*}
		\EE \norm{A_j}^4 &= \EE \left(\sum\limits_i \alpha_iy_{j,i}^2\right)^2 = \sum\limits_i \alpha_i^2\EE y_{j,i}^4 + \sum\limits_{i\neq k}\alpha_i\alpha_k\EE y_{j,i}^2y_{j,k}^2 \\&\leq 3\sum\limits_i\alpha_i^4 + \left(\sum\limits_i\alpha_i^2\right)^2 = 3\norm{\alpha}_4^4 +\norm{\alpha}_2^4,
		\end{align*}
		when we remember that the 4th moment of a centred Gaussian with variance $\alpha_i$ is $3\alpha_i^2$.
	\end{enumerate}
	We turn to bound each term of the sum \eqref{HS sum}:
	\begin{align*}
	&\EE_\mathbb{Y}\Bigg\vert\mathrm{tr}\left(I_n-\frac{1}{\norm{\alpha}_2^2}\mathbb{Y}D_\alpha\mathbb{Y}^T\right)\Bigg\vert \\
	\leq& 
	\sqrt{\EE_\mathbb{Y}\mathrm{tr}^2\left(I_n-\frac{1}{\norm{\alpha}_2^2}\mathbb{Y}D_\alpha\mathbb{Y}^T\right)} = \sqrt{\EE_\mathbb{Y}\left(\sum\limits_{j=1}^{n}\left(1 -\frac{\norm{A_j}^2}{\norm{\alpha}_2^2}\right)\right)^2}\\
	=& \sqrt{\EE_\mathbb{Y} \left(n^2 - \frac{2n}{\norm{\alpha}_2^2}\sum\limits_{j=1}^{n}\norm{A_j}^2 + \frac{1}{\norm{\alpha}_2^4}\sum\limits_{j \neq k}\norm{A_j}^2\norm{A_k}^2 + \frac{1}{\norm{\alpha}^4_2}\sum\limits_{j=1}^{n}\norm{A_j}^4\right)}\\
	\leq&\sqrt{n^2 - 2n^2 + 2\binom{n}{2} + \frac{n}{\norm{\alpha}_2^4}\left(3\norm{\alpha}_4^4 +\norm{\alpha}_2^4\right)} = \sqrt{3n\frac{\norm{\alpha}_4^4}{\norm{\alpha}_2^4}}.
	\end{align*}
	Similarly, we may deal with the second term:
	\begin{equation*}
	\left.\begin{aligned}
	&\EE_\mathbb{Y}\norm{I_n-\frac{1}{\norm{\alpha}_2^2}\mathbb{Y}D_\alpha\mathbb{Y}^T}^2_{HS}= \left(\sum\limits_{k,j}\frac{1}{\norm{\alpha}_2^4}\EE_\mathbb{Y}\left<A_j, A_k\right>^2\right)- n =\\ &\frac{1}{\norm{\alpha}_2^4}\sum\limits_{j=1}^{n}\EE_\mathbb{Y} \norm{A_j}^4 + \frac{1}{\norm{\alpha}_2^4}\sum\limits_{j\neq k} \left<A_j,A_k\right>^2 -n \leq\\&
	\frac{n}{\norm{\alpha}_2^4}(3\norm{\alpha}_4^4 + \norm{\alpha}_2^4) + \frac{2}{\norm{\alpha}_2^4}\binom{n}{2}\norm{\alpha}_4^4 - n =\\& 3n\frac{\norm{\alpha}_4^4}{\norm{\alpha}_2^4} +(n^2-n)\frac{\norm{\alpha}_4^4}{\norm{\alpha}_2^4} \leq 3n^2\frac{\norm{\alpha}_4^4}{\norm{\alpha}_2^4}.
	\end{aligned}\right.
	\end{equation*}
	Combining \eqref{HS sum} with the last two displays gives $$-\EE_\mathbb{Y}\left[ \ln\det\left(\frac{1}{\norm{\alpha}_2^2}\mathbb{Y}D_\alpha\mathbb{Y}^T\right)\mathbbm{1}_{\left\{\lambda_{\mathrm{min}}\geq \frac{1}{2}\right\}}\right] \leq 3\left(n^2 \left(\frac{\norm{\alpha}_4}{\norm{\alpha}_2}\right)^4 + \sqrt{n\left(\frac{\norm{\alpha}_4}{\norm{\alpha}_2}\right)^4 }\right).$$
\end{proof}
\begin{claim}
	$$\EE_\mathbb{Y}\left[ \ln\det\left(\frac{1}{\norm{\alpha}_2^2}\mathbb{Y}D_\alpha\mathbb{Y}^T\right)\mathbbm{1}_{\{\lambda_{\mathrm{min}} < 1/2\}}\right] < n\exp(-C\norm{\alpha}_2^2),$$
	for a universal constant $C>0$.
\end{claim}
\begin{proof}
	Observe that for any $\xi \in (0,\frac{1}{2})$:
	\begin{align}\label{integral-bound}
	-\EE_\mathbb{Y}\left[ \ln\det\left(\frac{\mathbb{Y}D_\alpha\mathbb{Y}^T}{\norm{\alpha}_2^2}\right)\mathbbm{1}_{\left\{\lambda_{\mathrm{min}} < 1/2\right\}}\right] &\leq n\EE\left(-\log(\lambda_{\mathrm{min}})\mathbbm{1}_{\left\{\lambda_{\mathrm{min}} < 1/2\right\}}\right)\nonumber\\
	&=n\int\limits_{\log(2)}^\infty\PP(-\log(\lambda_{\mathrm{min}})>t)dt\nonumber\\
	&=n\int\limits_0^{1/2}\frac{1}{s}\PP(\lambda_\mathrm{min}<s)ds\nonumber\\
	&\leq \frac{n}{\xi}\PP(\lambda_{\mathrm{min}}<1/2) + n\int\limits_0^\xi\frac{1}{s}\PP(\lambda_{\mathrm{min}}<s)ds.
	\end{align}
	By allowing $\xi$ to be some small constant, we'll need to bound $\PP(\lambda_{\mathrm{min}}<1/2)$ and $\PP(\lambda_{\mathrm{min}}<s)$ for small $s$.\\
	\\
	Recall that for any $s$, $\lambda_{\mathrm{min}} < s$ implies the existence of $\theta \in \mathbb{S}^{n-1}$ such that 
	\begin{equation*}
	\theta^T\frac{\mathbb{Y}D_\alpha\mathbb{Y}^T}{\norm{\alpha}^2_2}\theta < s \mbox{   , or equivalently   } \norm{\sqrt{D_\alpha}\mathbb{Y}^T\theta}^2 < s\norm{\alpha}^2_2.
	\end{equation*}
	Also, if $\theta$ is such that $\norm{\frac{\sqrt{D_\alpha}\mathbb{Y}^T}{\norm{\alpha}_2}\theta} < \sqrt{s}$, then for any $\theta' \in \mathbb{S}^{n-1}$, \begin{equation*}
	\norm{\frac{\sqrt{D_\alpha}\mathbb{Y}^T}{\norm{\alpha}_2}\theta'} < \sqrt{s} + \sqrt{\lambda_{\mathrm{max}}}\norm{\theta - \theta'},
	\end{equation*} 
	where $\lambda_{\mathrm{max}}$ is the largest eigenvalue of $\frac{\mathbb{Y}D_\alpha\mathbb{Y}^T}{\norm{\alpha}^2_2}$.\\
	\\
	We will first bound $\PP\left(\lambda_{\mathrm{min}}<1/2\right)$, using an $\varepsilon$-net argument.
	Note that for each $\theta$, $ \sqrt{D_\alpha}\mathbb{Y}^T\theta$ is distributed as $\mathcal{N}(0,D_\alpha^2)$. Consider the Euclidean metric on $\Sph$ and let $0<\eps<1$. We may cover $\Sph$ with $\left(\frac{3}{\varepsilon}\right)^n$  balls of radius $\varepsilon$ (see Lemma 2.3.4 in ~\cite{tao2012topics}, for example) to achieve
	\begin{equation} \label{eps-net}
	\PP\Bigg(\lambda_{\mathrm{min}}<1/2\Bigg)\leq \left(\frac{3}{\varepsilon}\right)^n\PP\left(\norm{\mathcal{N}(0,D_\alpha^2)} < \sqrt{ \frac{1.1}{2}\norm{\alpha}^2_2}\right) + \PP\left(\sqrt{\lambda_{\mathrm{max}}}>\frac{0.1}{\sqrt{2}\varepsilon}\right).
	\end{equation}
	\\
	To bound $\PP\left(\sqrt{\lambda_{\mathrm{max}}}>\frac{0.1}{\sqrt{2}\varepsilon}\right)$ we will use another $\eps$-net with $\eps = \frac{1}{2}$. Along with the fact that  $\norm{\theta -\theta'} \leq \frac{1}{2}$ implies $\norm{\frac{\sqrt{D_\alpha}\mathbb{Y}^T(\theta - \theta')}{\norm{\alpha}_2}} \leq \frac{\sqrt{\lambda_{\mathrm{max}}}}{2}$, we may see that
	\begin{align}\label{eps-net2}
	\PP\left(\sqrt{\lambda_{\mathrm{max}}}>\frac{0.1}{\sqrt{2}\varepsilon}\right) &\leq 6^n\PP\left(\norm{\sqrt{D_\alpha}\mathbb{Y}^T\theta}^2>\frac{0.01\norm{\alpha}_2^2}{4\varepsilon^2}\right)\nonumber\\
	&= 6^n\PP\left(\norm{\mathcal{N}(0,D_\alpha^2)} > \sqrt{\frac{0.01}{4\varepsilon^2}\norm{\alpha}^2_2}\right).
	\end{align}
	But, for any $x >0$:
	$$\PP\left(\norm{\mathcal{N}\left(0,D_\alpha^2\right)} > \sqrt{x\norm{\alpha}^2_2}\right) = \PP\left(\sum\limits_i\alpha_i^2\chi_i^2 >x\norm{\alpha}^2_2\right),$$
	where the $\chi_i^2$ are i.i.d. Chi-squared random variables with $1$ degree of freedom. Observe that $\EE[\alpha_i^2\chi_i^2] = \alpha_i^2$.\\
	\\
	We may now utilize the sub-exponential tail of the $\chi^2$ distribution and apply \eqref{berenstein type} with $v_i = \alpha_i^2$, noting that, by the normalization, \eqref{normalization}, $\norm{\alpha}_\infty=1$. Thus, provided that $x > 3$
	\begin{align}\label{exponential-bound}
	& \PP\left(\sum\alpha_i^2\chi_i^2 >\nonumber x\norm{\alpha}_2^2\right)\\\nonumber \leq & \PP\left(\left|\sum\alpha_i^2\chi_i^2 - \norm{\alpha}_2^2\right| > (x-1)\norm{\alpha}_2^2\right)\\
	\leq &
	2\exp\left(-\min\left(\frac{x-1}{2}\norm{\alpha}_2^2,\frac{(x-1)^2}{4}\norm{\alpha}_2^2\right)\right) \leq 2\exp\left(-\norm{\alpha}_2^2\right).
	\end{align}
	\\
	Substituting $x$ for $\frac{0.01}{4\varepsilon^2}$ in \eqref{eps-net2} shows that when $\frac{0.01}{4\varepsilon^2} > 3$ then 
	\begin{equation*}\label{lambdamax}
	\PP\left(\sqrt{\lambda_{\mathrm{max}}}>\frac{0.1}{\sqrt{2}\varepsilon}\right) \leq 6^n\exp(-\norm{\alpha}_2^2).
	\end{equation*}
	The exact same considerations as in \eqref{exponential-bound} also show that
	\begin{align*}
	&\PP\left(\norm{\mathcal{N}(0,D_\alpha^2)} < \sqrt{\frac{1.1}{2}\norm{\alpha}^2_2}\right)\\
	\leq&\PP\left(\left|\sum\limits_i\alpha_i^2\chi_i^2 -\norm{\alpha}_2^2\right| > \frac{0.9}{2}\norm{\alpha}_2^2 \right)\\
	\leq& 2\exp\left(-\frac{0.9^2}{16}\norm{\alpha}_2^2\right)\leq 2\exp\left(-\frac{\norm{\alpha}_2^2}{20}\right).
	\end{align*}
	\\
	Plugging the above two displays into \eqref{eps-net}, when $\varepsilon$ is small enough, yields
	\begin{equation}\label{boundlambdahalf}
	\PP(\lambda_{\mathrm{min}}<1/2) \leq 2\left(\frac{3}{\eps}\right)^ne^{-\frac{\norm{\alpha}_2^2}{20}} + 2\cdot 6^ne^{-\norm{\alpha}_2^2}\leq 4\exp\left(\frac{3n}{\eps}-\frac{\norm{\alpha}_2^2}{20}\right).
	\end{equation}
	For general $0 < s < 1/2$, in a similar fashion to \eqref{eps-net}, using an $s$-net gives the bound
	\begin{equation}\label{eps-net3}
	\PP\Big(\lambda_{\mathrm{min}}<s\Big)\leq \left(\frac{3}{s}\right)^n\PP\left(\norm{\mathcal{N}(0,D_\alpha^2)} < \sqrt{1.1s\norm{\alpha}^2_2}\right) + \PP\left(\sqrt{\lambda_{\mathrm{max}}}>0.1/\sqrt{s}\right).
	\end{equation}
	Now, $\mathcal{N}(0,D_\alpha^2)$ can be written as $D_\alpha Z_d$ where $Z_d$ is a standard Gaussian $d$-dimensional vector. In ~\cite[ Proposition 2.6]{latala2007banach}, it was shown that there exists universal constants $C_L,C'>0$ such that for any $t<C'$:
	\begin{equation*}\label{latala}
	\PP\Big(\norm{D_\alpha Z }< t\norm{D_\alpha}_{HS}\Big) \leq \exp\left(C_L\ln(t)\left(\frac{\norm{D_\alpha}_{HS}}{\norm{D_\alpha}_{op}}\right)^2\right)= \exp\left(C_L\ln(t)\norm{\alpha}_2^2\right) = t^{C_L\norm{\alpha}_2^2},
	\end{equation*}
	with equality stemming from the facts that $\norm{D_\alpha}_{HS} = \norm{\alpha}_2$ and $\norm{D_\alpha}_{op} = \norm{\alpha}_\infty=1$.
	Thus
	\begin{equation}\label{bound-gaussian}
	\PP\left(\norm{\mathcal{N}(0,D_\alpha^2)} < \sqrt{1.1s\norm{\alpha}^2_2}\right) \leq 2s^{\frac{C_L}{2}\norm{\alpha}_2^2}.
	\end{equation}
	By revisiting \eqref{eps-net2} and replacing $\sqrt{2}\eps$ with $\sqrt{s}$ we note that for small $s$
	\begin{align*}
	\PP\Big(\sqrt{\lambda_{\mathrm{max}}}>0.1/\sqrt{s}\Big) &\leq 6^n\PP\left(\norm{\mathcal{N}(0,D_\alpha^2)} > \sqrt{\frac{0.01}{2s}\norm{\alpha}^2_2}\right)\\
	&\leq 6^n\PP\left(\left|\sum\limits_i\alpha_i^2\chi_i^2 - \norm{\alpha}_2^2\right|>\left(\frac{0.01}{2s}-1\right)\norm{\alpha}_2^2\right).\\
	\end{align*}
	And, provided that $s \leq \frac{0.01}{4}$, \eqref{exponential-bound} shows
	\begin{equation}\label{lambdamaxs}
	\PP(\sqrt{\lambda_{\mathrm{max}}}>0.1/\sqrt{s}) \leq 6^n\exp\left(-\frac{1}{2s}\left(\frac{0.01}{2}-s\right)\norm{\alpha}_2^2\right)\leq 6^ne^{-\frac{0.01\norm{\alpha}_2^2}{4s}}.
	\end{equation}
	By using \eqref{lambdamaxs} and \eqref{bound-gaussian} to bound \eqref{eps-net3} we obtain
	\begin{equation*}
	\PP(\lambda_{\mathrm{min}}<s)\leq 2\left(\frac{3}{s}\right)^ns^{\frac{C_L}{2}\norm{\alpha}_2^2} + \exp\left(2n-\frac{0.01\norm{\alpha}_2^2}{4s}\right),\ \forall s \leq \frac{0.01}{4}.
	\end{equation*}
	We have thus shown, by combining \eqref{boundlambdahalf}, together with the last inequality into \eqref{integral-bound} and choosing $\xi$ to be a small enough constant:
	\begin{align*}
	&\frac{n}{\xi}\PP(\lambda_{\mathrm{min}}<1/2) + n\int\limits_0^\xi\frac{1}{s}\PP(\lambda_{\mathrm{min}}<s)ds \leq\\
	& \frac{n}{\xi}12\exp\left(\frac{3n}{\eps}-\frac{\norm{\alpha}_2^2}{20}\right) + n\int\limits_0^\xi 3^ns^{\frac{C_L}{2}(\norm{\alpha}_2^2 - n - 1)}+ \frac{1}{s}e^{\left(2n-\frac{0.01\norm{\alpha}_2^2}{4s}\right)} ds.
	\end{align*}
	Assuming that $\xi \leq \frac{1}{e}$ and that $\norm{\alpha}_2^2 > n + 1$,
	\begin{align*}
	& n\int\limits_0^\xi 3^ns^{\frac{C_L}{2}(\norm{\alpha}_2^2 - n - 1)}ds \leq n3^n \xi^{\frac{C_L}{2}(\norm{\alpha}_2^2 - n)} \leq ne^{\frac{C_L}{2}(n -\norm{\alpha}_2^2)+2n},\\
	& n\int\limits_0^\xi  \frac{1}{s}e^{\left(2n-\frac{0.01\norm{\alpha}_2^2}{4s}\right)} ds \leq ne^{2n}\int\limits_0^\xi e^{-\frac{0.01\norm{\alpha}_2^2}{8s}} ds \leq ne^{2n}\xi e^{-\frac{0.01\norm{\alpha}_2^2}{8\xi}}.
	\end{align*}
	To obtain the desired result we observe that if $n^3\left(\frac{\norm{\alpha}_4}{\norm{\alpha}_2}\right)^4 \to 0$ then $\left(\frac{\norm{\alpha}_2}{\norm{\alpha}_4}\right)^4 >> n^3$. the inequality $\norm{\alpha}_2^2 \geq \left(\frac{\norm{\alpha}_2}{\norm{\alpha}_4}\right)^{4/3}$ implies $\norm{\alpha}_2^2 >> n$,
	which shows the existence of a constant $C>0$ for which
	$$\EE_\mathbb{Y}\left[ \ln\det\left(\frac{1}{\norm{\alpha}_2^2}\mathbb{Y}D_\alpha\mathbb{Y}^T\right)\mathbbm{1}_{\{\lambda_{\mathrm{min}} < 1/2\}}\right] < n\exp(-C\norm{\alpha}_2^2).$$
\end{proof}
To finish the prove of Theorem \ref{thm: main}(b) we must now deal with $\mathrm{TV}\left(K_p(M(n)), H_{p,\alpha}(M(n))\right)$.
\begin{lemma} \label{lem: remainder bound}
Assume $n^3 \left(\frac{\norm{\alpha}_4}{\norm{\alpha}_2}\right)^4 \to 0$, then $\mathrm{TV}(K_p(M(n)), H_{p,\alpha}(M(n))) \to 0$.
\end{lemma}
\begin{proof}
First, we again pass to relative entropy using \eqref{pinsker}, Pinsker's inequality:
$$\mathrm{TV}(K_p(M(n)), H_{p,\alpha}(M(n)) \leq \sqrt{\mathrm{Ent}\left[K_p(M(n))|| H_{p,\alpha}(M(n))\right]}.$$
We note that both $K_p(M(n))$ and $H_{p,\alpha}(M(n))$ are simply Bernoulli matrices. The entries of $K_p(M(n))$ are i.i.d. $\mathrm{Bernoulli(p)}$, while the entries of $H_{p,\alpha}(M(n))$ are i.i.d. $\mathrm{Bernoulli}(p')$ where $p' = \Phi^{-1}\left(\frac{t_{p,\alpha}}{\norm{\alpha}_2}\right)$. Defining $\mathrm{Ent}[p||p'] := \mathrm{Ent}\left[\mathrm{Bernoulli}(p)||\mathrm{Bernoulli}(p')\right]$ and using the chain rule \eqref{chain rule} for relative entropy yields
$$\mathrm{Ent}\left[K_p(M(n))|| H_{p,\alpha}(M(n))\right] \leq n^2\mathrm{Ent}[p||p'].$$
One may verify that 
$$\lim\limits_{p'\to p}\frac{\mathrm{Ent}[p||p']}{(p-p')^2} = \lim\limits_{p'\to p} \frac{p\ln(\frac{p}{p'})+(1-p)\ln(\frac{1-p}{1-p'})}{(p-p')^2} = \frac{1}{2p-2p^2}.$$
So, $\frac{\mathrm{Ent}[p||p']}{(p-p')^2}$ is a continuous function on $(0,1) \times (0,1)$ and is bounded on every compact subset of its domain. Thus, there exists a constant $C_p$, depending on $p$ such that 
$$\mathrm{Ent}[p||p'] \leq C_p(p-p')^2.$$
By Lemma \ref{lem: tp}, $|p - p'| \leq 3\left(\frac{\norm{\alpha}_3}{\norm{\alpha}_2}\right)^3$,
which affords the bound
$$\mathrm{Ent}[p||p'] \leq 9C_p \left(\frac{\norm{\alpha}_3}{\norm{\alpha}_2}\right)^6.$$
But now, by Cauchy-Schwartz's inequality, $\norm{\alpha}_3^3 = \sum\limits_i\alpha_i\alpha_i^2\leq \sqrt{\norm{\alpha}_2^2\norm{\alpha}_4^4}$. Combining all of the above
\begin{align*}
&\mathrm{TV}\left(K_p(M(n)), H_{p,\alpha}(M(n))\right)^2 \leq \mathrm{Ent}[K_p(M(n))|| H_{p,\alpha}(M(n))]  \\
\leq & n^2\mathrm{Ent}(p||p') \leq 9C_pn^2 \left(\frac{\norm{\alpha}_3}{\norm{\alpha}_2}\right)^6 \leq n^2\frac{\norm{\alpha}^4_4\norm{\alpha}_2^2}{\norm{\alpha}^6_2} < n^3 \left(\frac{\norm{\alpha}_4}{\norm{\alpha}_2}\right)^4.
\end{align*}
\end{proof}

\bibliographystyle{acm}
\bibliography{bib}

\begin{thebibliography}{10}

\bibitem{bentkus2005lyapunov}
{\sc Bentkus, V.}
\newblock A lyapunov-type bound in {$R^d$}.
\newblock {\em Theory of Probability \& Its Applications 49}, 2 (2005),
  311--323.

\bibitem{BDER14}
{\sc Bubeck, S., Ding, J., Eldan, R., and R{\'a}cz, M.~Z.}
\newblock Testing for high-dimensional geometry in random graphs.
\newblock {\em Random Structures \& Algorithms 49}, 3 (2016), 503--532.

\bibitem{BG15}
{\sc Bubeck, S., and Ganguly, S.}
\newblock Entropic {CLT} and phase transition in high-dimensional {W}ishart
  matrices.
\newblock {\em International Mathematics Research Notices 2018}, 2 (2016),
  588--606.

\bibitem{cover2012elements}
{\sc Cover, T.~M., and Thomas, J.~A.}
\newblock {\em Elements of information theory}.
\newblock John Wiley \& Sons, 2012.

\bibitem{duchi2007derivations}
{\sc Duchi, J.}
\newblock Derivations for linear algebra and optimization.
\newblock {\em Berkeley, California\/} (2007).
\newblock {http://www.cs.berkeley.edu/˜jduchi/projects/general notes.pdf}.

\bibitem{durrett2010probability}
{\sc Durrett, R.}
\newblock {\em Probability: theory and examples}.
\newblock Cambridge university press, 2010.

\bibitem{eaton}
{\sc Eaton, M.~L.}
\newblock Chapter 8: The wishart distribution.
\newblock In {\em Multivariate Statistics}, vol.~53 of {\em Lecture
  Notes--Monograph Series}. Institute of Mathematical Statistics, Beachwood,
  Ohio, USA, 2007, pp.~302--333.

\bibitem{ER60}
{\sc Erd{\H{o}}s, P., and R{\'e}nyi, A.}
\newblock On the evolution of random graphs.
\newblock {\em Publ. Math. Inst. Hungar. Acad. Sci 5\/} (1960), 17--61.

\bibitem{hormander2015analysis}
{\sc H{\"o}rmander, L.}
\newblock {\em The analysis of linear partial differential operators I:
  Distribution theory and Fourier analysis}.
\newblock Springer, 2015.

\bibitem{latala2007banach}
{\sc Latala, R., Mankiewicz, P., Oleszkiewicz, K., and Tomczak-Jaegermann, N.}
\newblock Banach-{M}azur distances and projections on random subgaussian
  polytopes.
\newblock {\em Discrete \& Computational Geometry 38}, 1 (2007), 29--50.

\bibitem{nourdin2012normal}
{\sc Nourdin, I., and Peccati, G.}
\newblock {\em Normal approximations with Malliavin calculus: from Stein's
  method to universality}, vol.~192.
\newblock Cambridge University Press, 2012.

\bibitem{P95}
{\sc Petrov, V.~V.}
\newblock {\em Limit Theorems of Probability Theory}.
\newblock Oxford University Press, 1995.

\bibitem{S91}
{\sc Shephard, N.~G.}
\newblock From characteristic function to distribution function: a simple
  framework for the theory.
\newblock {\em Econometric theory 7}, 04 (1991), 519--529.

\bibitem{tao2012topics}
{\sc Tao, T.}
\newblock {\em Topics in random matrix theory}, vol.~132.
\newblock American Mathematical Society Providence, RI, 2012.

\bibitem{V12}
{\sc Vershynin, R.}
\newblock Introduction to the non-asymptotic analysis of random matrices.
\newblock In {\em Compressed sensing}. Cambridge Univ. Press, Cambridge, 2012,
  pp.~210--268.

\end{thebibliography}

\end{document}